\documentclass[10pt]{siamltex}
\usepackage[twoside, a4paper, lmargin=3.5cm, rmargin=3.5cm, top=2.5cm, bottom=2.5cm]{geometry} 

\usepackage{amsmath}
\usepackage{amssymb}
\usepackage{euscript}

\usepackage{amsmath}
\usepackage{amssymb}
\usepackage{euscript}
\usepackage{graphicx}
\usepackage{color}
\definecolor{dgreen}{rgb}{0,.6,0}

\title{Hyperbolic cross approximation in infinite dimensions}
\author{ 
Dinh D\~ung \\[2mm]
 Information Technology Institute, Vietnam National University, Hanoi, \\
144 Xuan Thuy, Cau Giay, Hanoi, Vietnam\\
{\ttfamily dinhzung@gmail.com}\\[5mm] 
Michael Griebel \\[2mm]
Institute for Numerical Simulation, Bonn University \\
Wegelerstrasse 6, 53115 Bonn, Germany \\
Fraunhofer Institute for Algorithms and Scientific Computing SCAI\\
Schloss Birlinghoven, 53754 Sankt Augustin, Germany \\
{\ttfamily griebel@ins.uni-bonn.de} 
}

\def\ZZ{{\mathbb Z}}
\def\ZZdp{{\mathbb Z}^d_+}

\def\II{{\mathbb I}}

\def\DD{{\mathbb D}}
\def\NN{{\mathbb N}}
\def\RR{{\mathbb R}}
\def\GG{{\mathbb G}}

\def\TT{{\mathbb T}}

\def\DDd{{\mathbb D}^d}

\def\TTd{{\mathbb T}^d}
\def\TTm{{\mathbb T}^m}

\def\ZZd{{\mathbb Z}^d}
\def\ZZm{{\mathbb Z}^m}
\def\IIm{{\mathbb I}^m}
\def\ZZmp{{\mathbb Z}^m_+}
\def\IIi{{\mathbb I}^\infty}
\def\DDi{{\mathbb D}^\infty}

\def\RRi{{\mathbb R}^\infty}
\def\TTi{{\mathbb T}^\infty}

\def\RRi{{\mathbb R}^\infty}

\def\ZZi{{\mathbb Z}^{\infty}}
\def\ZZi{{\mathbb Z}^{\infty}}
\def\ZZis{{\mathbb Z}^{\infty}_*}
\def\ZZip{{\mathbb Z}^\infty_+}
\def\ZZips{{\mathbb Z}^\infty_{+*}}
\def\NNis{{\mathbb N}^\infty_*}

\def\Gg{{\mathcal G}}

\def\Hh{{\mathcal H}}
\def\Ii{{\mathcal I}}
\def\Jj{{\mathcal J}}
\def\Jjd{{\mathcal J}_d}
\def\Ll{{\mathcal L}}
\def\Pp{{\mathcal P}}
\def\Ss{{\mathcal S}}
\def\Tt{{\mathcal T}}
\def\Uu{{\mathcal U}}

\def\supp{\operatorname{supp}}

\newcommand{\bk}{{\bf k}}
\newcommand{\bl}{{\bf l}}

\newcommand{\bn}{{\bf n}}

\newcommand{\br}{{\bf r}}
\newcommand{\bs}{{\bf s}}

\newcommand{\bx}{{\bf x}}
\newcommand{\by}{{\bf y}}
\newcommand{\bb}{{\bf b}}

\newlength{\fixboxwidth}
\begin{document}
\maketitle
\begin{abstract}
We give tight upper and lower bounds of the cardinality of the
index sets of certain hyperbolic crosses  which reflect mixed Sobolev-Korobov-type smoothness and mixed Sobolev-analytic-type smoothness in the infinite-dimensional case where specific summability properties of the smoothness indices are fulfilled. 
These estimates are then applied to the linear approximation of functions from the associated spaces in terms of the $\varepsilon$-dimension of their unit balls. Here, the approximation is based on linear information. 
Such function spaces appear for example for the solution of parametric and stochastic PDEs.
The obtained upper and lower bounds of the approximation error as well as of the associated
$\varepsilon$-complexities are completely independent of any parametric or stochastic dimension. Moreover, the rates are independent of
 the parameters which define the smoothness properties of the infinite-variate parametric or stochastic part of the solution. These parameters are only contained in the order constants. 
This way, linear approximation theory becomes possible in the infinite-dimensional case and corresponding infinite-dimensional problems get tractable.  


\end{abstract}

\section{Introduction}

The efficient approximation of a function of infinitely many variables is an important issue in a lot of applications in physics, finance, engineering and statistics. 
It arises in uncertainty quantification, computational finance and computational physics 
and is encountered for example in the numerical treatment of path integrals, stochastic processes, random fields and stochastic or parametric PDEs.
While the problem of quadrature of functions in weighted Hilbert spaces with infinitely many variables has recently found a lot of interest 
in the information based complexity community, see e.g. \cite{BG14,CDMR09,DG14a,DG14b,G10,GMR14,HMNR,HW,KSWW, NHMR,PW,PWW,SW98,WW96}, there is much less literature on approximation.  
So far, the approximation of functions in weighted Hilbert spaces with infinitely many variables has been studied for a properly weighted $L_2$-error norm\footnote{There is also \cite{Wa12a,WW11a,WW11b}, where however a norm in a special Hilbert space was employed such that the approximation problem indeed got easier than the integration problem.} in \cite{Wa12b}. 
In any case, a reproducing kernel Hilbert space $H_K$ with kernel $K=\sum_u \gamma_u k_u$ is employed with $|u|$-dimensional kernels $k_u$ where $u$ varies over all finite subsets of $\NN$. It involves a sequence of weights  $\gamma = (\gamma_u)$ that moderate the influence of terms which depend on the variables associated with the finite-dimensional index sets $u \subset \{1,2,\ldots\}=\NN$. Weighted spaces had first been suggested for the finite-dimensional case in \cite{SW98}, see also \cite{SW1,SW2}. For further details, see \cite{DKS} and the references cited therein. The approximation of functions with anisotropically weighted Gaussian kernels has been studied in \cite{FHW}.

Moreover, there is work on sparse grid integration and approximation, see  \cite{BG} for a survey and bibliography.
It recently has found applications in uncertainty quantification for stochastic and parametric PDEs, 
especially for non-intrusive methods, compare \cite{BNTT,BNTT2,CCDS,CCS,CDS10b,CDS10a,GWZ,HS14,JR13,LK,NTW1,NTW2}. 
There, for the stochastic or parametric part of the problem, 
an anisotropic sparse grid approximation or quadrature is constructed either a priori from knowledge of the covariance decay of the stochastic data or a posteriori by means of dimension-adaptive refinement. 
This way, the infinite-dimensional case gets truncated dynamically to finite dimensions while the higher dimensions are trivially resolved by the constant one.
Here, in contrast to the above-mentioned approach using a weighted reproducing kernel Hilbert space, one usually relies on spaces with smoothness of increasing order, either for the mixed Sobolev smoothness situation or for the analytic setting. Thus, as already noticed in 
\cite{PW10}, one may have two options for obtaining tractability: either by using decaying weights or by using increasing smoothness. 

Sparse grids and hyperbolic crosses promise to break the curse of dimensionality which appears for conventional approximation methods, at least to some extent.
However, the approximation rates and cost complexities of conventional sparse grids for isotropic mixed Sobolev regularity still involve logarithmic terms which grow exponentially with the dimension.
In \cite{Di13, GH13} it could be shown that the rate of the approximation error and the cost complexity 
get completely independent of the dimension for the case of anisotropic mixed Sobolev regularity with sufficiently fast rising smoothness indices.  This also follows from results on approximation for anisotropic mixed smoothness, see, e.g., \cite{Di86, Te86} for details.
But the constants in the bounds for the approximation error and the cost rate could not be estimated sharply and still depend on the dimension $d$. 
Therefore, this result can not straightforwardly be extended to the infinite-dimensional case, i.e. to the limit of $d$ going to $\infty$.

This will be achieved in the present paper. 
To this end, we rely on the infinite-variate space $\Hh$ which is the tensor product ${\cal H}:=H^\alpha(\GG^m) \otimes K^\br(\DDi)$ of the Sobolev space 
$H^\alpha(\GG^m)$ and the infinite-variate space 
$K^\br(\DDi)$ of mixed smoothness with varying Korobov-type smoothness indices ${\bf r}=r_1, r_2 ,\ldots $, or we rely on the tensor product ${\cal H}:=H^\alpha(\GG^m) \otimes A^{\br,p,q} (\DDi)$ of $H^\alpha(\GG^m)$ with the
infinite-variate space $A^{\br,p,q}(\DDi)$ of mixed smoothness with varying analytic-type smoothness indices 
${\bf r}=r_1, r_2 ,\ldots $ (and $p$ and $q$ entering algebraic prefactors). The approximation error is measured in the tensor product Hilbert space $\Gg:=H^\beta(\GG^m) \otimes L_2(\DDi,\mu)$
with $\beta \ge 0$, which is isomorphic to the Bochner space $L_2(\DDi,H^\beta(\GG^m))$. 
Here, $\GG$ denotes either the unit circle (one-dimensional torus) $\TT$ in the periodic case or the interval 
$\II:= [-1,1]$ in the nonperiodic case.
Furthermore, $\DD$ is either $\TT$, $\II$ or $\RR$, depending on the respective situation under consideration.
Altogether, $\GG^m$ denotes the $m$-fold (tensor-product) domain where the $m$-dimensional physical coordinates live, whereas $\DDi$ denotes the infinite-dimensional (tensor-product) domain where the infinite-dimensional stochastic or parametric coordinates live.
Moreover,
$L_2(\DDi):=L_2(\DDi, \mu)$ is the space of all 
infinite-variate functions $f$ on $\DDi$ such that 
$\int_{\DDi} |f(\by)|^2 \, \mathrm{d} \mu(\by)
\ < \ \infty
$
with the infinite tensor product probability measure $\mathrm{d} \mu$ which is based on properly chosen univariate probability measures. 

Here, the spaces $H^\alpha(\GG^m) \otimes K^\br(\DDi)$ generalize the usual $d$-variate Korobov spaces $K^{(r_1,\ldots,r_d)}(\DDd)$ of mixed smoothness with different smoothness indices $r_1,\ldots,r_d$ to the infinite-variate case and additionally
contain in a tensor product way also the Sobolev space $H^\alpha(\GG^m)$. Moreover, the 
$(m+d)$-variate Sobolev-Korobov-type spaces $\Hh_d:= H^\alpha(\GG^m) \otimes K^\br(\DDd)$ of mixed smoothness with different weights for arbitrary but finite $d$ are naturally contained. 
Similarly,
the spaces $H^\alpha(\GG^m) \otimes A^{\br,p,q}(\DDi)$ generalize the $d$-variate spaces of mixed analytic smoothness with smoothness indices $p,q$ and $r_1,\ldots,r_d$ (for a precise definition, see Subsection \ref{Index sets for analytic smoothness}) to the infinite-variate case and additionally
contain in a tensor product way the Sobolev space $H^\alpha(\GG^m)$. 
Moreover, the 
$(m+d)$-variate Sobolev-analytic-type spaces $\Hh_d:= H^\alpha(\GG^m) \otimes A^{\br,p,q}(\DDd)$ of mixed Sobolev and analytic smoothnesses with different weights for arbitrary $d$ are naturally contained.
Thus, the problem of approximating functions from $\Hh$ in the $\Gg=H^\beta(\GG^m) \otimes L_2(\DDi)-$norm 
directly governs the problem of approximating functions from $\Hh_d$
 in the $\Gg_d:=H^\beta(\GG^m) \otimes L_2(\DDd)-$norm in both cases, where $d$ may be large but finite.

Now, let us fix some notation which is needed to describe the cost complexity of an approximation. In general, 
if $X$ is a Hilbert space and $W$ a subset of $X$, 
the Kolmogorov $n$-width $d_n(W,X)$ \cite{Ko36} is given by
\begin{equation} \label{[d_n]}
d_n(W,X)
:= \ 
\inf_{{M_n}} \ \sup_{f \in W} \ \inf_{g \in {M_n}} \|f - g\|_X,
\end{equation} 
where the outer infimum is taken over all linear manifolds {$M_n$} in $X$ of
dimension at most $n$.\footnote{ 
A different worst-case setting is represented by the linear $n$-width $\lambda_n(W,X)$ \cite{Ti60}.
This corresponds to a characterization of the best linear approximation error, see, e.g., \cite{DU13} for definitions.
Since $X$ is here a Hilbert space, both concepts coincide, i.e., we have
$d_n(W,X) \ = \ \lambda_n(W,X).$
}
Furthermore, the so-called $\varepsilon$-dimension $n_\varepsilon = n_\varepsilon(W,X)$ is usually employed to quantify the computational complexity. 
It is defined as 
\[
n_\varepsilon(W,X)
:= \ 
{
\inf \left\{n:\, \exists  \ M_n: 
\ \dim M_n \le n, \ \sup_{f \in W} \ \inf_{g \in M_n} \|f - g\|_X \le \varepsilon \right\}
},
\] 
where $M_n$ is a linear manifold in $X$ of dimension $\le n$.
This quantity is just the inverse of $d_n(W,X)$. 
Indeed, $n_\varepsilon(W,X)$ is the minimal number
$n_\varepsilon$ such that the approximation of $W$ by a suitably
chosen $n_\varepsilon$-dimensional subspace $M_n$ in $X$ yields the approximation error to be less or equal to $\varepsilon$. 
Moreover, $n_\varepsilon(W,X)$ is the smallest number of linear functionals that is needed by an algorithm to give for all $f \in W$ an approximation
with an error of at most $\varepsilon$.
From the computational point of view it is more convenient to study $n_\varepsilon(W,X)$ than $d_n(W,X)$ since it is directly related to the computational cost.

The approximation of functions with anisotropic mixed smoothness goes back to papers by various authors from the former Soviet Union, initiated in \cite{Ba60}. We refer the reader to \cite{Di86, Te86} for a survey and bibliography. In particular, in \cite{Di84}, the rate of the cardinality of anisotropic hyperbolic crosses was computed and used in the estimation of $d_n(U^\br(\TTd),L_2(\TTd))$, where $U^\br(\TTd)$ is the unit ball of the space of functions of
bounded anisotropic mixed derivatives $\br$ with respect to the  $L_2(\TTd)$-norm. 
Moreover, the $\varepsilon$-dimensions of classes of mixed smoothness were investigated in \cite{Di79,Di80,DD91,DM79}. 
Recently, $n$-widths and $\varepsilon$-dimensions in the classical isotropic Sobolev space $H^r(\TTd)$ of $d$-variate periodic functions and of Sobolev classes with  mixed and other anisotropic smoothness have been studied for high-dimensional settings \cite{CD13, DU13}.
There,  although the dimension $n$ of the approximating subspace is the main parameter of the convergence rate, where $n$ is going to infinity, 
the parameter $d$ may seriously affect this rate when $d$ is large. 

Now, let $\Uu$ and $\Uu_d$  be the unit ball in $\Hh$ and $\Hh_d$, respectively. In the present paper, we give upper and lower bounds for  $n_\varepsilon(\Uu,\Gg)$ and $n_\varepsilon(\Uu_d,\Gg_d)$ for both, the Sobolev-Korobov and the Sobolev-analytic mixed smoothness spaces. 
To this end, we first derive tight estimates of the cardinalities of hyperbolic cross index sets  which are associated to the chosen accuracy $\varepsilon$. 
Since the corresponding approximations in infinite tensor product Hilbert spaces then possess the accuracy $\varepsilon$,
this indeed gives bounds on $n_\varepsilon(\Uu,\Gg)$ and $n_\varepsilon(\Uu_d,\Gg_d)$. 
Here, depending on the underlying domain, we focus on 
approximations by trigonometric polynomials (periodic case) and Legendre and Hermite polynomials (nonperiodic case, bounded and non-bounded domain)
with indices from hyperbolic crosses that correspond to frequencies and powers in the infinite-dimensional case. 
Altogether, we are able to show estimates that are completely independent on any dimension $d$ in both, rates and order constants, provided that a moderate summability condition on the sequence of smoothness indices $\bf r$ holds.

In the following, as example, let us mention one of our main results on the cardinality of hyperbolic crosses in the infinite-dimensional case and on the related $\varepsilon$-dimension.
To this end, let $m, \alpha, \beta, \br, p,q$ be given by  
\begin{equation} \label{[br]}
m \in \NN; \quad \alpha > \beta \ge 0; \quad 
\br = (r_j)_{j=1}^\infty \in \RR^\infty_+, \ 0 < r_1 \le r_2 \le \cdots \le r_j\cdots; \quad p\ge 0, \ q >0.  
\end{equation} 
(with the additional restriction $(\alpha - \beta)/m < r_1$ for the Sobolev-Korobov smoothness case).
Then, if moderate summability conditions on the sequence of smoothness indices $\bf r$ hold, we have for every 
$d \in \NN$ and every $\varepsilon \in (0,1]$
\begin{equation} \label{ineq[<n_e<]}
\lfloor \varepsilon^{- 1/(\alpha - \beta)} \rfloor^m - 1
\ \le \
n_\varepsilon(\Uu_d, \Gg_d)
\ \le \
n_\varepsilon(\Uu, \Gg)
\ \le \
C \, \varepsilon^{- m/(\alpha - \beta)},
\end{equation}
where  $C$ depends on $m, \alpha, \beta, \br, p,q$ but not on $d$. Thus, the upper and lower bounds on the $\varepsilon$-dimension
$n_\varepsilon(\Uu_d, \Gg_d)$ are completely free of the dimension for any value of $d$.
These estimates are derived from the relations
\begin{equation} \nonumber
|G(\varepsilon^{-1})| - 1
\ \le \
n_\varepsilon(\Uu, \Gg)
\ \le \
|G(\varepsilon^{-1})|,
\end{equation}
and 
\begin{equation} \label{Introduction[G(T)]}
\lfloor T^{1/(\alpha - \beta)} \rfloor^m 
\ \le \
|G(T)|
\ \le \
C \, T^{m/(\alpha - \beta)},
\end{equation}
where $G(T)$ is the relevant hyperbolic cross induced by $T=\varepsilon^{-1}$
and  $C$ is as in \eqref{ineq[<n_e<]}. Throughout this paper, $|G|$ denotes the cardinality of a finite set $G$.

Note here the following properties of \eqref{ineq[<n_e<]}:
\begin{itemize}
\item[(i)] The upper and lower bounds of $n_\varepsilon(\Uu_d, \Gg_d)$  and $n_\varepsilon(\Uu, \Gg)$ are tight and independent of $d$.
\item[(ii)] The term $\varepsilon^{- m/(\alpha - \beta)}$ depends on $\varepsilon$, on the dimension $m$ and on the smoothnesses indices $\alpha$ and $\beta$ of the Sobolev component parts $H^\alpha(\GG^m)$ and $H^\beta(\GG^m)$ of the spaces $\cal H$ and $\cal G$ only. 
\item[(iii)] The infinite series of smoothness indices $\br$ in the Korobov or analytic component parts is contained in $C$ only and does not show up in the rates.
\item[(iv)] The necessary summability conditions on $m, \alpha, \beta, \br, p, q$  are natural and quite moderate (see e.g. the assumptions of Theorem \ref{theorem[n_e]}). 
\end{itemize}
Altogether, the right-hand side of \eqref{ineq[<n_e<]} is (up to the constant) just the same as the rate which we would obtain 
from an approximation problem  in the energy norm  
 $H^\beta(\GG^m)$ of functions from the Sobolev space $H^\alpha(\GG^m)$ alone, i.e. $n_\varepsilon(U^\alpha(\GG^m), H^\beta(\GG^m)) \asymp \varepsilon^{- m/(\alpha - \beta)}$, where 
$U^\alpha(\GG^m)$ is the unit ball in $H^\alpha(\GG^m)$. Therefore, the $\varepsilon$-dimension and thus the complexity of
our infinite-dimensional approximation problem is (up to the constant) just that of the $m$-dimensional approximation problem without the
infinite-variate component. 

From \eqref{ineq[<n_e<]} we can see that the $d$-dimensional problem $n_\varepsilon(\Uu_d, \Gg_d)$ is strongly polynomially tractable. This property crucially depends on the chosen norm as well on certain restrictions on the mixed smoothness, see for instance the prerequisites of Theorem~\ref{theorem[G<]} and of Theorem~\ref{theorem[n_e]-periodic}  in the case of mixed Sobolev-Korobov-type smoothness. If a different type of norm definition is employed, completely different results are obtained. For an example, see \cite{PW10}, where it is shown that increasing the smoothness (no matter how fast) does not help there.
Moreover, the effect of different norm definitions on the approximation numbers of Sobolev embeddings with particular emphasis on the dependence on the dimension is studied for the periodic case in \cite{KMU15,KSU14,KSU15}.

For an application of our approximation error estimate, let us now assume that an infinite-variate function $u$ is living in  $\Gg$ which is isomorphic 
to the Bochner space $ L_2(\DD^\infty,H^{\beta}(\GG^m) )$. Let us furthermore assume that $u$ possesses some higher regularity, i.e., to be precise, assume 
that $u \in \Hh$ with either $\Hh=H^{\alpha}(\GG^m) \otimes K^{\br}(\DD^\infty)$ or $\Hh=H^{\alpha}(\GG^m) \otimes A^{\br,p,q}(\DD^\infty)$.
Let $n=|G(T)|$ and $L_n$ be the projection onto the subspace of suitable polynomials with frequencies or/and powers in $G(T)$. 
Then, with the above notation and assumptions for $\varepsilon$-dimensions, we have 
\[
\|u - L_n(u)\|_\Gg 
\ \leq \ 
C^{(\alpha - \beta)/m} \,  
n^{- (\alpha - \beta)/m} \, \| u\|_\Hh,
\]
where $C$ is as in \eqref{ineq[<n_e<]}.
Indeed, such functions $u$ are, for certain $\sigma(\bx,\by)$ and $f(\bx,\by)$,  the solution
of the parametric or stochastic elliptic PDE
\begin{equation}\label{Introduction[spde]}
-\mathrm{div}_\bx(\sigma(\bx,\by)\nabla_\bx u(\bx,\by)) = f(\bx,\by) \quad \bx \in \GG^m \quad \by \in \DD^\infty,
\end{equation} 
with homogeneous boundary conditions $u(\bx,\by)=0, \ \bx\in \GG^m $, $\by \in \DD^\infty$.
Here, we have to find a real-valued function $u: \GG^m \times \DD^\infty \to \RR$ such that 
\eqref{Introduction[spde]} holds $\mu$-almost everywhere.
Thus, our results also shed light on the $\varepsilon$-dimension and the complexity of properly defined linear approximation schemes for infinite-dimensional stochastic/parametric PDEs in the case of linear information.

The remainder of this paper is organized as follows: 
In Section \ref{cardinality of HC}, we establish tight upper and lower bounds for the cardinality of hyperbolic crosses with varying smoothness weights in the infinite-dimensional setting. 
In Section \ref{Approximation in Hilbert spaces}, we study
hyperbolic cross approximations and their $\varepsilon$-dimensions in infinite tensor product Hilbert spaces.
In Section \ref{Function approximation}, we consider hyperbolic cross approximations of infinite-variate functions in particular for periodic functions from periodic Sobolev-Korobov-type mixed spaces and for nonperiodic functions from the Sobolev-analytic-type mixed smoothness spaces.
In Section \ref{application}, we discuss the application of our results to a model problem from parametric/stochastic elliptic PDEs. Finally, we give some concluding remarks in Section \ref{conclusion}.

\section{The cardinality of  hyperbolic crosses in the infinite-dimensional case}
\label{cardinality of HC}

In this section, we establish upper and lower bounds for the cardinality of various index sets of hyperbolic crosses
in the infinite-dimensional case.
First, we consider index sets which correspond to the mixed Sobolev-Korobov-type setting with varying polynomial smoothness, then we consider cases with mixed Sobolev-analytic smoothness where also exponential smoothness terms appear. 

We will use the following notation: $\RRi$ is the set of all sequences $\by= (y_j)_{j=1}^\infty$ with $y_j \in \RR$ and
$\ZZi$ is the set of all sequences $\bs= (s_j)_{j=1}^\infty$ with $s_j \in \ZZ$. Furthermore, 
$\ZZip := \{\bs \in \ZZi: s_j \ge 0, \ j =1,2,...\}$, $y_j$ is the $j$th coordinate of 
$\by \in \RR^\infty$. Moreover, $\ZZis$ is a subset of $\ZZi$ of all $\bs$ such that $\supp(\bs)$ is finite, where 
$\supp(\bs)$ is the support of $\bs$, that is the set of all $j \in \NN$ such that 
$s_j \not=0$. Finally, $\ZZips := \ZZis \cap \ZZip$.

\subsection{Index sets for mixed Sobolev-Korobov-type smoothness}   
\label{Index sets for Korobov smoothness} 

Let $m\in \ZZ_+$, $a >0$ be given and let $\br = (r_j)_{j=1}^\infty \in \RR^\infty_+$ be given with
\begin{equation}  \nonumber
 \quad 0<r = r_1 \cdots = r_{t+1}; \quad  r_{t+2} \le  r_{t+3} \le \cdots.
\end{equation}
For each  $(\bk,\bs) \in \ZZm \times \ZZis$ with $\bk \in \ZZm$ and $\bs \in \ZZis$, we define the scalar $\lambda(\bk,\bs)$ by
\begin{equation} \label{[lambda_{br}]}
\lambda(\bk,\bs) 
 := \
\max_{1\le j \le m}(1 + |k_j|)^a\prod_{j=1}^{t+1}(1 + |s_j|)^r 
\prod_{j=t+2}^\infty(1 + |s_j|)^{r_j}. 
\end{equation}
Here, the associated functions will possess isotropic smoothness of index $a$ for the coordinates $k_j$, they will possess  mixed smoothness of index $r$ for
the first $t+1$ coordinates $s_j$ of $\bf s$ and they will possess monotonously rising mixed smoothness indices $r_j$ for the following coordinates $s_j$.

Now, for $T>0$, consider the hyperbolic crosses in the infinite-dimensional setting $\ZZmp \times \ZZips$ with indices $a,\br$ 
\begin{equation} \label{G(T)}
G(T) 
:= \ 
\big\{(\bk,\bs)  \in \ZZmp \times \ZZips: \lambda(\bk,\bs)   \leq T\big\}.
\end{equation}
The cardinality of  $G(\varepsilon^{-1})$ will later describe the necessary dimension of the approximation spaces of the associated linear approximation with accuracy $\varepsilon$.

We want to derive estimates for the cardinality of $G(T)$.
To this end, we will make use of the following lemmata.

\begin{lemma} \label{lemma[sum-convex-function]}
Let $\mu, \nu \in \NN$ with $\mu \ge \nu$, and let $\varphi$ be a convex function on the interval $(0,\infty)$. 
Then, we have
\begin{equation}
\sum_{k=\nu}^\mu \varphi(k)
\ \le \
\int_{\nu - 1/2}^{\mu + 1/2} \varphi(x) \, \mathrm{d} x,
\quad \mbox{and} \quad 
\sum_{k=\nu}^\infty \varphi(k)
\ \le \
\int_{\nu - 1/2}^\infty \varphi(x) \, \mathrm{d} x,
\end{equation}
\end{lemma}
\begin{proof}
Observe that, for a convex function $g$ on $[0,1]$, there holds true the inequality 
\begin{equation} 
g(1/2)
\ \le \
\frac{1}{2}\int_0^1  [g(x) + g(1-x)] \, \mathrm{d} x 
\ = \ \int_0^1 g(x) \, \mathrm{d} x.
\end{equation}
Applying this inequality to the functions $g_k(x):= \varphi(x + k- 1/2)$, $k \in \NN$, we obtain
\begin{equation} 
\varphi(k)
\ = \
g_k(1/2)
\ \le \
 \int_0^1 g_k(x) \, \mathrm{d} x
\ = \ \int_{k-1/2}^{k+1/2} \varphi(x) \, \mathrm{d} x.
\end{equation}
Hence, we have for any $\mu, \nu \in \NN$ with $\mu \ge \nu$,
\begin{equation}\label{kserest}
\sum_{k=\nu}^\mu \varphi(k)
\ \le \
\sum_{k=\nu}^\mu  \int_{k-1/2}^{k+1/2} \varphi(x) \, \mathrm{d} x
\ = \
\int_{\nu - 1/2}^{\mu + 1/2} \varphi(x) \, \mathrm{d} x
\end{equation}
which proves the first and, with $\mu \to \infty$, the second inequality of the lemma.
\hfill\end{proof}

\begin{lemma} \label{lemma[<sum_k^-a]}
Let $\eta>1$. Then, we have
\begin{equation} \label{ineq[<sum_k^-a]]}
\sum_{k=2}^\infty k^{-\eta}
\ < \
 \frac{1}{\eta-1} \left(\frac{3}{2}\right)^{-(\eta-1)}.
\end{equation}
\end{lemma}
\begin{proof}
Since the function $x^{-\eta}$ is non-negative and convex on $(0,\infty)$, we obtain
\begin{equation}\nonumber
\sum_{k=2}^\infty k^{-\eta}
\ \le \
\int_{3/2}^\infty x^{-\eta} \, \mathrm{d} x 
\ = \
\frac{1}{\eta-1} \left(\frac{3}{2}\right)^{-(\eta-1)}
\end{equation}
by applying Lemma \ref{lemma[sum-convex-function]} for $\nu = 2$.
\hfill\end{proof}

\begin{lemma} \label{lemsumfin}
Let $\eta\in \RR\setminus (0,1)$ and $\mu \in \NN$. Then, we have 
\begin{equation} \label{ineqsumfin}
\sum_{k=1}^\mu k^{\eta}
\ \le \
\begin{cases}
\frac{1}{\eta+ 1} \left(\frac{3}{2}\right)^{\eta+ 1} \mu^{\eta+ 1}, \  & \ \eta> -1, \\[2ex]
\log (2\mu + 1), \  & \ \eta= -1, \\[2ex]
 \frac{1}{|\eta| - 1} 2^{|\eta| - 1}, \  & \ \eta< -1.
\end{cases}
\end{equation}
\end{lemma}
\begin{proof}
The function $x^\eta$ is non-negative and convex in the interval $(0, \infty)$ for $\eta\in \RR \setminus (0,1)$. Thus, there holds
by Lemma \ref{lemma[sum-convex-function]}
\begin{equation} \label{[sum_s^alpha<(2)]} 
\sum_{k=1}^\mu k^{\eta}
\ \le \
\int_{1/2}^{\mu+1/2} x^{\eta} \, \mathrm{d} x
\ = \ 
\begin{cases}
\frac{1}{\eta+ 1} x^{\eta+ 1} \Big|_{1/2}^{\mu+1/2}, \  & \ \eta\not= -1, \\[4ex]
\log x \Big|_{1/2}^{\mu+1/2}, \  & \ \eta= -1.
 \end{cases}
\end{equation}
Hence, 
\begin{equation} \label{[sum_s^alpha<(1)]}
\sum_{k=1}^\mu k^\eta
\ \le \
\begin{cases}
\frac{1}{\eta+ 1}(\mu+1/2)^{\eta+ 1}, \  & \ \eta> -1, \\[2ex]
\log (\mu+1/2) + \log 2, \  & \ \eta= -1, \\[2ex]
\frac{1}{|\eta| - 1} 2^{|\eta| - 1}, \  & \ \eta< -1,
 \end{cases}
\end{equation}
which implies (\ref{ineqsumfin}). 
\hfill\end{proof}

First of all, for $T\ge 1$, $m \in \NN$, we consider the hyperbolic cross in the finite, $m$-dimensional case
\begin{equation}\nonumber
\Gamma(T) 
:= \ 
\bigg\{\bl \in \NN^m: \prod_{j=1}^m l_j \leq T\bigg\}.
\end{equation}
We have the following bound for the cardinality $ |\Gamma(T)|$:
\begin{lemma} \label{lemma[Gamma<]}
For $T \ge 1$, it holds
\begin{equation} \label{[Gamma<](2)}
|\Gamma(T)|
\ \le \
 \frac{2^m}{(m-1)!} \  T(\log T + m \log 2)^{m-1}.
\end{equation}
\end{lemma}

\begin{proof}
Observe that $|\Gamma(T)|=|\Gamma^*(T)|$, where
\begin{equation}\label{Gammas}
\Gamma^*(T) 
:= \ 
\bigg\{\bl \in \ZZ^m_+: \prod_{j=1}^m(1 + l_j) \leq T\bigg\}.
\end{equation}
Thus, we need to derive an estimate for $|\Gamma^*(T)|$. 
From \cite[Lemma 2.3, Corollary 3.3]{CD13}, it follows\footnote{There, the set $\{\bl \in \ZZ^m: \prod_{j=1}^m(1 + l_j) \leq T\}$ is analyzed. Since our 
$\Gamma^*(T)$ is a subset, (\ref{[Gamma_+<](1)}) clearly holds. A direct bound for (\ref{Gammas}) might be possible without 
the multiplier $2^m$, but only for $T=T(m)$ large enough. See  \cite{CD13} for details.}
that for every $T \ge 1$,
\begin{equation} \label{[Gamma_+<](1)}
{|\Gamma^*(T)|} 
\ \le \
 \frac{2^m}{(m-1)!} \  \frac{T(\log T + m \log 2)^m}{\log T + m \log 2 + m-1}.
\end{equation}
Since $m\geq 1$ we immediately obtain the desired estimate.
\hfill\end{proof}

Now, for given $m,t \in \NN$ and $a,r >0$, put
\begin{equation} \label{BBB}
B(a,r,m,t) 
:= \
\begin{cases}
\left(\frac{3}{2}\right)^m \,
\left(1 + \frac{1}{rm/a - 1}\left(\frac{3}{2}\right)^{-(rm/a - 1)} \right)^t , \  & \ r > a/m, \\[2ex]
m\, \frac{2^{t + 1}}{t!},  \  & \ r = a/m,  \\[2ex]
m\, \frac{2^{t + 1}}{t!} (a/r - m)^{-1} \, 2^{a/r - m}, \  & \ r < a/m.
 \end{cases}
\end{equation}
Furthermore, define
\begin{equation}\label{AAA}
A(a,r,m,t,T) \ := \
\begin{cases}
T^{m/a}, \  & \ r > a/m, \\[2ex]
T^{m/a} \log (2T^{1/a} + 1)  [r^{-1}\log T + (t + 1) \log 2]^t , \  & \ r = a/m,\\[2ex]
T^{1/r} [r^{-1}\log T + (t + 1) \log 2]^t , \  & \ r < a/m.
 \end{cases}
\end{equation}

Next, for $T \ge 1$, we consider the hyperbolic crosses in the finite, $(m+t+1)$-dimensional case 
\begin{equation} \label{H(T)}
H(T) 
:= \ 
\bigg\{\bn\in \NN^{m + t + 1}: \max_{1\le j \le m}n_j^a\prod_{j=m+1}^{m+t+1}n_j^r \leq T\bigg\}.
\end{equation}
We have the following bound for the cardinality $|H(T)|$:
\begin{lemma} \label{lemma[H<]}
For $T \ge 1$, it holds
\begin{equation} \label{[H<](1)}
|H(T)|
\ \le \
B(a,r,m,t)  \, A(a,r,m,t,T) 
\end{equation}
\end{lemma}
\begin{proof} 
Let us introduce the following notation for convenience: For $\bn\in \NN^{m+t+1}$, we write $\bn=(\bn', \bn^{''})$ with $\bn'=(n_1,...,n_m)$ and 
$\bn^{''}=(n_{m+1},...,n_{m+t+1})$. We then represent $\NN^{m+t+1}$ as 
\[
\NN^{m+t+1} \, = \, \NN^m \times \NN^{t+1}.
\]

We first consider the case $r \le a/m$.
Note that, for every $\bn\in {H(T)}$, we have that  
$n_j \le T^{1/a}, \ j=1,...,m$.  
For a $\bn' \in \NN^m$ put  
\[
T_{\bn'} \ = \ \left(T(\max_{1\le j \le m}n_j)^{-a}\right)^{1/r}.
\] 
Hence, by symmetry of the variables, we have for every $T \ge 1$,
\begin{equation} 
\begin{split}
{|H(T)}| 
&= \
\sum_{\substack{\bn': \ n_j \le T^{1/a}
\\ j=1,...,m}} \ 
\sum_{\bn^{''}: \ \prod_{j=m+1}^{m+ t + 1}n_j \ \leq \ T_{\bn'} \,} 1 \
\le \ 
m \,  \sum_{\substack{\bn': \ n_m \le T^{1/a} 
\\ n_j \le n_m, \ j=1,...,m-1}} \ 
\sum_{\prod_{j=m+1}^{m+ t + 1}n_j \ \leq \ \left(T n_m^{-a}\right)^{1/r}\,} 1.
\end{split}
\end{equation}
The inner sum in the last expression is $|\Gamma(T_{\bn'})|$ from (\ref{Gammas}) with $m = t + 1$.
The application of Lemma \ref{lemma[Gamma<]} then gives for every $T \ge 1$
\begin{equation} \label{[|H^*(T)|<]}
\begin{split}
{|H(T)}| 
&\le \ 
m \,  \sum_{\substack{\bn': \ n_m \le T^{1/a}
 \\ n_j \le n_m, \ j=1,...,m-1}} \ 
|\Gamma(T_{\bn'})| \\[2ex]
&\le \ 
m \,  \sum_{\substack{\bn': \ n_m \le T^{1/a} 
\\n_j \le n_m, \ j=1,...,m-1}} \ 
 \frac{2^{t + 1}}{t!} \  T_{\bn'}\left(\log T_{\bn'} + (t + 1) \log 2\right)^t \\[2ex]
&\le \ 
m \,  \sum_{\substack{\bn': \ n_m \le T^{1/a} 
\\n_j \le n_m, \ j=1,...,m-1}} \ 
\frac{2^{t + 1}}{t!} \,\left(T n_m^{-a}\right)^{1/r} \left(\log ((T n_m^{-a})^{1/r}) + (t + 1) \log 2\right)^t  \\[2ex]
&\le \ 
m \, \frac{2^{t + 1}}{t!} \, T^{1/r} \left(\log(T^{1/r}) + (t + 1) \log 2\right)^t 
 \sum_{n_m \le T^{1/a}} \ \sum_{\substack{n_j \le n_m \\ j=1,...,m-1}} \ 
 n_m^{-a/r}   \\[2ex]
&\le \ 
m \, \frac{2^{t + 1}}{t!} \, T^{1/r} \left(\log(T^{1/r}) + (t + 1) \log 2\right)^t 
\sum_{n_m \le T^{1/a}} \ n_m^{m-1-a/r}.
\end{split}
\end{equation}
Now, by applying Lemma \ref{lemsumfin}
 with $k = n_m$, $\eta= m-1-a/r$ and $\mu = \lfloor T^{1/a} \rfloor$ to the  sum in the last line of \eqref{[|H^*(T)|<]}, we derive 
\begin{equation} \label{[sum_k_n]}
\sum_{n_m \le T^{1/a}} \
 n_m^{m-1-a/r}   
\ \le \
\begin{cases}
\log (2\lfloor T^{1/a}\rfloor + 1),  & \ r = a/m, \\[2ex]
(a/r - m)^{-1} \, 2^{a/r - m}, \  & \ r < a/m,
 \end{cases}
\end{equation}
which, together with \eqref{[|H^*(T)|<]} and the definitions (\ref{BBB}) and (\ref{AAA}), proves \eqref{[H<](1)} in the case $r \leq a/m$.

Let us now consider the case $ r > a/m$. For a $\bn^{''} \in \NN^{t + 1}$ we put  
\[
T_{\bn^{''}} \ = \ T^{1/a}\prod_{j=m+ 1}^{m+t+1}n_j^{-r/a}.
\]
Note that, for every $\bn\in H(T)$, we have that 
$n_j \le T^{1/r}, \ j=m+1,...,m+ t + 1$.  
By symmetry of the variables, we have for every $T \ge 1$,
\begin{equation} 
\begin{split}
{|H(T)}| 
\ &= \
\sum_{\substack{\bn^{''}: \ n_j \le T^{1/r} \\ j=m+1,...,m+ t + 1}} \ 
\sum_{\bn': \ \max_{1\le j \le m}n_j \ \leq \ T_{\bn^{''}} \,} 1 \\[2ex]
&\le \ 
m \,  \sum_{\substack{\bn^{''}: \ n_j \le T^{1/r} \\ j=m+1,...,m+ t + 1}} \ 
\sum_{\substack{\bn': \ n_m \le T_{\bn^{''}} \\ n_j \le n_m, \ j=1,...,m-1}} 1 \\[2ex] 
&\le \ 
m \,  \sum_{\substack{\bn^{''}: \ n_j \le T^{1/r} \\ j=m+1,...,m+ t + 1}} \ 
\sum_{n_m \le T_{\bn^{''}}} n_m^{m-1} \ .
\end{split}
\end{equation}
From  \eqref{[sum_s^alpha<(1)]} and the inequality $T_{\bn^{''}} \ge 1$ it follows that
\begin{equation} 
\sum_{n_m \le T_{\bn^{''}}} n_m^{m-1}
\ \le \
\frac{1}{m} (T_{\bn^{''}} + 1/2)^m
{\ \le \
\frac{1}{m} (T_{\bn^{''}} + T_{\bn^{''}}/2)^m}
\ \le \
\frac{1}{m} \left(\frac{3}{2}\right)^m T^{m/a}\prod_{j=m+ 1}^{m+t+1}n_j^{-rm/a}.
\end{equation}
Hence,
\begin{equation}\label{Estimations}
\begin{split}
{|H(T)}| 
\ &\le \ 
\left(\frac{3}{2}\right)^m T^{m/a} \sum_{\bn^{''}: \ n_j \le T^{1/r}, \ j=m+1,...,m+ t + 1} \ 
\prod_{j=m+ 1}^{m+t+1}n_j^{-rm/a} \\[2ex]
\ &= \
\left(\frac{3}{2}\right)^m T^{m/a} \prod_{j=m+ 1}^{m+t+1}\sum_{n_j \le T^{1/r}} \ 
n_j^{-rm/a} \\[2ex]
 \ &\le \
\left(\frac{3}{2}\right)^m T^{m/a} \prod_{j=m+ 1}^{m+t+1}\left(1 + \sum_{n_j=2}^\infty \ 
n_j^{-rm/a}\right) \\[2ex]
\ & < \
\left(\frac{3}{2}\right)^m T^{m/a} \prod_{j=m+ 1}^{m+t+1}
\left(1 + \frac{1}{rm/a - 1}\left(\frac{3}{2}\right)^{-(rm/a - 1)} \right) \\[2ex]
\ &= \
\left(\frac{3}{2}\right)^m T^{m/a}\,
\left(1 + \frac{1}{rm/a - 1}\left(\frac{3}{2}\right)^{-(rm/a - 1)} \right)^t. 
\end{split}
\end{equation}
With the definitions (\ref{BBB}) and (\ref{AAA}), this proves \eqref{[H<](1)} in the case $r > a/m$.
\hfill\end{proof}

We will frequently use the following well-known bound for infinite products: Let  $(p_j)_{j=1}^\infty$ be a summable sequence of positive numbers, that is 
$ \sum_{k=1}^\infty p_k \ < \ \infty. $
Then,  we have
\begin{equation} \label{condition[P]}
 \prod_{k=1}^\infty (1+p_k)
\ \le \ 
\exp \left( \sum_{k=1}^\infty p_k\right).
\end{equation}

Now we are in the position to state the main result of this section, i.e. we will give explicit bounds on the cardinality of the infinite-dimensional hyperbolic cross $G(T)$ defined in \eqref{G(T)}.
To this end, let the triple $m,a,\br$ be given by 
\begin{equation} \label{m,a,br[K]}
m\in \NN; \quad a >0; \quad 
\br = (r_j)_{j=1}^\infty, \ 0<r = r_1 \cdots = r_{t+1}, \quad  r_{t+2} \le  r_{t+3} \le \cdots.
\end{equation}
We put
\begin{equation} \label{[lambda]}
\lambda
:= \
\max\{m/a,1/r\}
\end{equation}
and define, for a nonnegative integer $t$, the terms
\begin{equation} \label{condition[M]}
M(t):= \ \sum_{j=t + 2}^\infty \, \frac{1}{\lambda\, r_j - 1} \left(\frac{3}{2}\right)^{-(\lambda\, r_j-1)}
\end{equation} 
and
\begin{equation}\label{C(a,r,m,t)} 
C(a,r,m,t) 
:= \
\, e^{M(t)} \, B(a,r,m,t).
\end{equation}

\begin{theorem} \label{theorem[G<]}
Assume that $\lambda \, r_{t+2} > 1$ and $M(t) < \infty$. 
Then, we have for every $T \ge 1$
\begin{equation} \label{[G_+<](1)}
\begin{aligned}
|G(T)|
\ &\le \
C(a,r,m,t) \, A(a,r,m,t,T) \\[2ex]
\ &= \ C(a,r,m,t) 
\begin{cases}
T^{m/a}, \  & \ r > a/m, \\[2ex]
T^{m/a} \log (2T^{1/a} + 1)  [r^{-1}\log T + (t + 1) \log 2]^t , \  & \ r = a/m,\\[2ex]
T^{1/r} [r^{-1}\log T + (t + 1) \log 2]^t , \  & \ r < a/m.
 \end{cases}
\end{aligned}
\end{equation}
\end{theorem}

\begin{proof}
Let us show, for example, the last inequality in \eqref{[G_+<](1)} where $r < a/m$.  Let $T \ge 1$ be given.
Observe that 
\begin{equation} \label{Gamma=G}
|G(T)| \ = \ |G^*(T)|,
\end{equation}
where
\begin{equation} \label{[Gt]}
G^*(T) 
:= \ 
\bigg\{(\bn,\bs') \in \NN^{m+t+1} \times \NNis(t): \max_{1\le j \le m} n_j^a \prod_{j=m+1}^{m + t+1} n_j^r 
\prod_{j=t+2}^\infty s_j^{r_j} \leq T\bigg\},
\end{equation}
and $\NNis(t)$ denotes the set of all indices $\bs'=(s_{t+2},s_{t+3},...)$ with $s_j \in \NN$ such that the set of 
$j \ge t+2$ with $s_j \not= 1$ is finite. Here, $G^*$ builds with $(\bn,\bs') \in \NN^{m+t+1} \times \NNis(t)$ on a different index splitting than $G$ from (\ref{G(T)}) with
$(\bk,\bs)  \in \ZZmp \times \ZZips$. This will allow for an easier decomposition later on. But their cardinalities are the same.

Note that, for every $(\bn,\bs') \in G^*(T)$,  it follows from the definition 
\begin{equation} \label{[H(T,n,nu)]}
H(T,t)
:= \ \left\{\bs' \in \NNis(t): \, \prod_{j=t+2}^\infty s_j^{r_j} 
\ \le \ T\right\}
\end{equation}
that $\bs' \in H(T,t)$.
Hence, 
\begin{equation}
|G^*(T)| 
\ = \
\sum_{\bs' \in H(T,t)} \quad 
\sum_{\bn\in \NN^{m+t+1}: \ 
(\max_{1\le j \le m} n_j)^a \prod_{j=m+1}^{m + t+1} n_j^r \ \le \ T\prod_{j=t+2}^\infty s_j^{-r_j}} 1.
\end{equation}
We now fix a $\bs' \in \NNis(t)$ for the moment, and put  
\[
T_{\bs'} \ = \ T\prod_{j \in J(\bs')} s_j^{-r_j},
\]
where $J(\bs'):= \{j \in \NN: j \ge t + 2, \ s_j \not= 1\}$. Note that $J(\bs')$ is a finite set by definition.
We then have $T_{\bs'} \ge 1$ and
\begin{equation} \label{[Gamma=]}
|G^*(T)| 
\ = \
\sum_{\bs' \in H(T,t)}
|{H (T_{\bs'})}|,
\end{equation}
where ${H (T_{\bs'})}$ is defined as in \eqref{H(T)}.
For ease of notation, we now use the abbreviation $B:= B(a,r,m,t)$. 
Lemma \ref{lemma[H<]} for the case $r < a/m$ gives
\begin{equation}
|{H (T_{\bs'})}|
\ \le \
B \, A(a,r,m,t,T_{\bs'})
\ = \ 
B \,  T_{\bs'}^{1/r}\, [r^{-1}\log T_{\bs'} + (t + 1) \log 2]^t. 
\end{equation}
Hence,  by \eqref{[Gamma=]} and with $\bs' \in \NNis(t)$, we have for every $T \ge 1$ that there holds
\begin{equation} \label{[Gamma<]}
\begin{split}
|G^*(T)| 
\ &\le \
\sum_{\bs' \in H(T,t)} 
B \,  T_{\bs'}^{1/r}\, [r^{-1}\log T_{\bs'} + (t + 1) \log 2]^t \\[2ex]
&= \
B \, \sum_{\bs' \in H(T,t)} \quad 
\left[T\prod_{j \in J(\bs')} s_j^{-r_j}\right]^{1/r}
\left[r^{-1} \log \left(T\prod_{j \in J(\bs')} s_j^{-r_j}\right) + (t + 1) \log 2 \right]^t \\[2ex]
& \le \
B  \, T^{1/r} 
[r^{-1} \log T + (t + 1) \log 2]^t \sum_{\bs' \in H(T,t)}  \quad 
 \prod_{j \in J(\bs')} s_j^{- \lambda r_j},
\end{split}
\end{equation}
where we used the inequality $\prod_{j \in J(\bs')} s_j^{-r_j} \le 1$ in the last step. 
Next, let us estimate the sum in the last line. Note that, for every $\bs' \in H(T,t)$,  it follows from definition \eqref{[Gt]} that 
$s_j \le T^{1/r_j}, \ j \ge  t + 2$. The condition $M(t) < \infty$ yields that $r_j \to \infty$ as $j \to \infty$. Hence, there exists $j^*$ such that $s_j = 1$ for every $\bs' \in H(T,t)$ and every $j > j^*$. 
From this observation we get
\begin{equation} \label{sum[prod]<}
\begin{split}
\sum_{\bs' \in H(T,t)}  \ 
 \prod_{j \in J(\bs')} s_j^{- \lambda r_j} 
\ &\le \
\sum_{s_j \le T^{1/r_j}, \ j \ge  t + 2} \ 
 \prod_{j= t + 2}^{j^*} s_j^{- \lambda r_j} 
 \\[2ex]
\ &\le  \
 \prod_{j= t + 2}^{j^*} \ 
\sum_{s_j \le T^{1/r_j}, \ j \ge  t + 2}  s_j^{- \lambda r_j} 
\ \le \ 
\prod_{j= t + 2}^{j^*} \ \sum_{s_j=1}^\infty 
 s_j^{- \lambda r_j}.
\end{split}
\end{equation}
The application of Lemma \ref{lemma[<sum_k^-a]} gives 
\begin{equation}
\sum_{s_j=1}^\infty 
 s_j^{- \lambda r_j}
 \ = \
 1 \ + \ \sum_{s_j=2}^\infty 
 s_j^{- \lambda r_j}
 \ < \
1 \ + \ \frac{1}{\lambda r_j-1} \left(\frac{3}{2}\right)^{-(\lambda r_j-1)} .
\end{equation}
Hence, by \eqref{condition[P]}, we get the inequality
\begin{equation} \nonumber
\sum_{\bs' \in H(T,t)}  \quad 
 \prod_{j=t+2}^\infty s_j^{- \lambda r_j} 
\ \le \
\prod_{j= t + 2}^\infty  
\left[1 + \frac{1}{\lambda r_j-1} \left(\frac{3}{2}\right)^{-(\lambda r_j-1)} \right] \\[2ex]
\ \le  \
 e^{M(t)},
\end{equation}
which, together with \eqref{[Gamma<]}, proves the theorem for the case $r < a/m$. The other cases can be shown in a similar way.
\hfill\end{proof}

\subsection{Index sets for mixed Sobolev-analytic-type smoothness} 
\label{Index sets for analytic smoothness} 
 
In this subsection we consider hyperbolic crosses of a different type with the replacement of  $\lambda(\bk,\bs)$ by $\rho(\bk,\bs)$ in its definition, where  $\rho(\bk,\bs)$ now also involves exponentials in the $\bs$-dependent part.
Such situations appear for example in applications with elliptic stochastic PDEs from uncertainty quantification.
There, the arising functions are analytic for the stochastic/parametric $\bf y$-part of the coordinates where, in many cases, a product structure still appears.
In the $\bx$-coordinates we will keep our previous Sobolev-type setting, but in the $\by$-coordinates we will switch to exponential terms involving certain weights in the exponents. Furthermore, the exponential terms are multiplied by a polynomial prefactor. 
To this end, let the  $5$-tuple $m,a,\br,p,q$ be given by
 \begin{equation} \label{m,a,br,p,q[A]}
m\in \NN; \quad a >0; \quad 
\br = (r_j)_{j=1}^\infty, \  r_j > 0, \ j \in \NN; \quad p\ge 0, \ q > 0.
\end{equation}
For each  $(\bk,\bs) \in \ZZm \times \ZZis$, we define the scalar $\rho(\bk,\bs)$ by
\begin{equation}\label{[lambda{a,mu,br}]}
\rho(\bk,\bs)
 := \
 \max_{1\le j \le m}(1 + |k_j|)^a \, \prod_{j=1}^\infty (1 + p|s_j|)^{-q}\, \exp \left((\br,|\bs|)\right), \quad 
(\br,|\bs|):= \sum_{j=1}^\infty r_j|s_j|.
\end{equation}
Note that polynomial prefactors of the type $(1+p|s_j|)^{-q}$ indeed appear with values $p=2$ and $q=1/2$ in estimates for the coefficients of Legendre expansions in specific situations for stochastic elliptic PDEs under the assumption of polyellipse analytic regularity, see e.g. \cite{BNTT2, Chkifathesis,CDS10b, CDS10a}. Also, more general forms of prefactors exist, involving factorials or the gamma function. Furthermore, one may consider individual values $p_j,q_j$ for each $j$. For reasons of simplicity, we stick to the case of constant $p,q$.
Here and in what follows, $|\bs|:=(|s_j|)_{j=1}^\infty$ for $\bs \in \ZZis$.

Now, for $T>0$, consider the following hyperbolic cross in the infinite-dimensional case
\begin{equation} \label{E(T)}
E(T) 
:= \ 
\big\{(\bk,\bs) \in \ZZmp \times \ZZips: \rho(\bk,\bs) \leq T\big\}.
\end{equation}
Again, the cardinality of  $E(\varepsilon^{-1})$ will later describe the necessary dimension of the approximation spaces of the associated linear approximation with accuracy $\varepsilon$.
We want to derive an estimate for the cardinality of $E(T)$.

\begin{theorem} \label{theorem[E<]}
Let  $m \in \NN$, $a > 0$, $\br = (r_j)_{j=1}^\infty \in \RRi$  with $r_j > 0$ and let $p=0$, $q\geq 0$. 
Assume that there holds the condition
\begin{equation}\label{M00}
M_{0,q}(m):=\sum_{j = 1}^\infty \frac 1 {e^{m r_j/a}-1} \ < \ \infty.
\end{equation}
Then we have for every $T \ge 1$
\begin{equation} \label{[E<]}
|E(T)|
\ \le \
\left(\frac 3 2 \right)^{2m} \exp[M_{0,q}(m)] \ T^{m/a}.
\end{equation}
\end{theorem}
\begin{proof}
Let $T \ge 1$ be given. Observe that 
$
|E(T)| \ = \ |E^*(T)|,
$
where
\begin{equation} \nonumber
E^*(T) 
:= \ 
\big\{(\bk,\bs) \in \NN^m \times \ZZips: \rho^*(\bk,\bs) \leq T\big\},
\end{equation}
and 
\begin{equation} \label{[lambda{a,mu,br}^*]}
\rho^*(\bk,\bs)
 := \
\max_{1\le j \le m}k_j^a \, \exp \left(\sum_{j=1}^\infty r_js_j\right) 
\ = \ \max_{1\le j \le m}k_j^a  \, \exp \left((\br,\bs)\right) . 
\end{equation}
Thus, we need to derive an estimate for $|E^*(T)|$. 
To this end, for $\bs \in \ZZips$, we put  
\[
T_{\bs} \ := \ T^{1/a} \exp [-(\br,\bs)/a].
\]
Note that, for every $(\bk,\bs) \in E^*(T)$,  it follows from the definition of $E^*(T)$ that $\bs \in H(T)$ 
where
\begin{equation} \label{[H(T)]}
H(T)
:= \ \left\{\bs \in \ZZips: \, \exp \left((\br,\bs)\right)\ \le \ T\right\}.
\end{equation}
By definition and the symmetry of the variable $k_j$ we have
\begin{equation}
\begin{split}
|E^*(T)| 
\ &= \
\sum_{\bs \in H(T)} \quad 
\sum_{\bk \in \NN^m: \ 
\max_{1\le j \le m} k_j \ \le \ T_{\bs}} 1 
\  \le \
m \sum_{\bs \in H(T)} \quad 
\sum_{\substack{\bk \in \NN^m: \ 
k_m \ \le \ T_{\bs} \\ k_j \le k_m, \ j = 1,...,m-1}} 1 \\[2ex]
\ & \le \
m \sum_{\bs \in H(T)} \quad 
\sum_{k \in \NN: \ k \ \le \ T_{\bs}} k^{m-1}.
\end{split}
\end{equation}
Hence,  since $T_{\bs} \ge 1$, the application of Lemma \ref{lemsumfin} gives
\begin{equation}
\begin{split}
|E^*(T)| 
\ &\le \
 m \sum_{\bs \in H(T)} \frac{1}{m} \left(\frac{3}{2}\right)^m (T_{\bs} + 1/2)^m  
\ \le \
 m \sum_{\bs \in H(T)} \frac{1}{m} \left(\frac{3}{2}\right)^m (T_{\bs} + T_{\bs}/2)^m  \\[2ex]
\ &\le \
\left(\frac{3}{2}\right)^m  \, \sum_{\bs \in H(T)} \left(\frac{3}{2}\right)^m \,T_{\bs}^m  
\ \le \
\left(\frac{3}{2}\right)^{2m} \, \sum_{\bs \in H(T)} T_{\bs}^m  \\[2ex]
\ &= \
\left(\frac{3}{2}\right)^{2m} \, T^{m/a} \, \sum_{\bs \in H(T)} \exp [- m (\br,\bs)/a].  
\end{split}
\end{equation}
Note furthermore that, for every $\bs \in H(T)$,  it follows from definition \eqref	{E(T)} that 
$s_j \le (\log T)/r_j, \ j \in \NN$. The condition $M_{0,q}(m) < \infty$ yields that 
$r_j \to \infty$ as $j \to \infty$. Hence, there exists a $j^*$ such that $s_j = 0$ for every $\bs \in H(T)$ and every $j > j^*$. 
From this observation we get
\begin{equation}
\begin{split}
\sum_{\bs \in H(T)} \exp [- m (\br,\bs)/a]  
\ &\le \
\prod_{j=1}^{j^*} \sum_{s_j \le (\log T)/r_j} \exp (- m r_j s_j/a) 
 \\
\ &\le \
\prod_{j=1}^{j^*} \left[1 + \sum_{s_j = 1}^\infty \exp (- m r_j s_j/a)\right] 
\ = \
\prod_{j=1}^{j^*} \left[1 + (e^{m r_j/a}-1)^{-1}\right].
\end{split}
\end{equation}
By use of the inequality \eqref{condition[P]} we finally obtain the bound 
\begin{equation}
|E^*(T)| 
\ \le \
 \left(\frac{3}{2}\right)^{2m} \, T^{m/a}  \, 
\exp \left[\sum_{j = 1}^\infty (e^{m r_j/a}-1)^{-1}\right],
\end{equation}
which gives the desired result.
\hfill\end{proof}

In order to estimate $E(T)$ for the case $p > 0$, we need an auxiliary lemma.

\begin{lemma} \label{lemma[<sum_exp]}
Let $\eta>0$ and let $ p,q >0$. Assume that 
\begin{equation} \label{[assumption-exp]}
\eta \ \ge \ \frac{q + \sqrt{q}}{p} \quad \mbox{and} \quad \eta > pq.
\end{equation}
Then, we have
\begin{equation} \label{ineq[<sum_exp]]}
\sum_{k=1}^\infty e^{-\eta k}(1 + pk)^q
\ \le \
\frac{(1 + p/2)^q }{\eta - pq} \, e^{- \eta/2}.
\end{equation}
\end{lemma}
\begin{proof}
Set $\varphi(x):= e^{-\eta x}(1 + px)^q$. Let us first show that the assumption \eqref{[assumption-exp]} provides the convexity of $\varphi$ on $(0,\infty)$. 
Changing variables by putting $t = 1 + px$, we have 
$\varphi(x)= e^{\delta}g(t)$, where 
$\delta:= \eta/p$ and $g(t):= e^{- \delta t}\, t^q$. Observe that the function $g$ is convex for 
$t \ge \frac{q + \sqrt{q}}{\delta}$. Since the change of variables is linear and $p > 0$, this implies that the function $\varphi$ is convex for $x \ge \frac{q + \sqrt{q}}{\eta} - \frac{1}{p}$. Hence, for 
$\eta\ge \frac{q + \sqrt{q}}{p}$, $\varphi$ is a non-negative convex function in $(0,\infty)$.

By applying Lemma \ref{lemma[sum-convex-function]}, we have
\begin{equation} \nonumber
\sum_{k=1}^\infty e^{-\eta k}(1 + pk)^q
\ \le \
\int_{1/2}^\infty e^{-\eta x}(1 + px)^q \, \mathrm{d} x 
\ = \
\frac{e^{\eta/p}}{p}\int_{1 + p/2}^\infty e^{-\eta t/p} \, t^q \,\mathrm{d} t. 
\end{equation}
Observing that 
\begin{equation} \nonumber
\int_{1 + p/2}^\infty e^{-\eta t/p} \, t^q \, \mathrm{d}t 
\ \ge \
\int_{1 + p/2}^\infty e^{-\eta t/p} \, t^{q-1} \, \mathrm{d}t,
\end{equation}
the integral in the last expression can be estimated as
\begin{equation} \nonumber
\begin{split}
\int_{1 + p/2}^\infty e^{-\eta t/p} \, t^q \, \mathrm{d}t 
\ &= \
\frac{p}{\eta} \,  e^{-(\eta/p)(1+p/2)}\, (1 + p/2)^q 
\ + \ 
\frac{pq}{\eta} \,\int_{1 + p/2}^\infty e^{-\eta t/p} \, t^{q-1} \, \mathrm{d}t 
\\[2ex]
\ &\le \
\frac{p}{\eta} \,  e^{-(\eta/p)(1+p/2)}\, (1 + p/2)^q 
\ + \ 
\frac{pq}{\eta} \,\int_{1 + p/2}^\infty e^{-\eta t/p} \, t^q \, \mathrm{d}t. 
\end{split}
\end{equation}
Hence,
\begin{equation} \nonumber
\int_{1 + p/2}^\infty e^{-\eta t/p} \, t^q \, \mathrm{d}t 
\ \le \
\frac{p(1 + p/2)^q }{\eta - pq} \,  e^{-(\eta/p)(1+p/2)} 
\end{equation}
Summing up we obtain
\begin{equation} \nonumber
\sum_{k=1}^\infty e^{-\eta k}(1 + pk)^q
\ \le \
\frac{(1 + p/2)^q }{\eta - pq} \,  e^{- \eta/2} \\[2ex]
\end{equation}
which proves the lemma.
\hfill\end{proof}

\begin{theorem} \label{theorem[E<](2)}
Let  $m \in \NN$, $a > 0$, $\br = (r_j)_{j=1}^\infty \in \RRi$  with $r_j > 0$ and let $p,q >0$. 
Assume that there hold the conditions
\begin{equation} \label{[assumption-exp-theorem]}
r_j \ \ge \  \frac{q + \sqrt{  qa/m}}{p}\, \quad \quad \quad 
r_j \ > \  pq \, \quad j \in \NN,
\end{equation}
and
\begin{equation}\label{Mpq}
M_{p,q}(m):= \ (1 + p/2)^{q m/a} \sum_{j=1}^\infty \, \frac{ e^{- m r_j/2a}}{m r_j/a - pqm/a} \,  \ < \ \infty.
\end{equation}
Then we have for every $T \ge 1$,
\begin{equation} \label{[E<](2)}
|E(T)|
\ \le \
\left(\frac 3 2 \right)^{2m} \exp[M_{p,q}(m)] \ T^{m/a}.
\end{equation}
\end{theorem}
\begin{proof}
As in the proof of Theorem \ref{theorem[E<]}, observe that 
$
|E(T)| \ = \ |E^*(T)|,
$
where
\begin{equation} \nonumber
E^*(T) 
:= \ 
\big\{(\bk,\bs) \in \NN^m \times \ZZips: \rho^*(\bk,\bs) \leq T\big\},
\end{equation}
and 
\begin{equation} \label{[lambda{a,mu,br}^*](2)}
\rho^*(\bk,\bs)
 := \
\max_{1\le j \le m}k_j^a  \, \prod_{j=1}^\infty (1 + p s_j)^{-q}\, \exp \left((\br,\bs)\right). 
\end{equation}
Thus, we need to derive an estimate for $|E^*(T)|$. 
To this end, for $\bs \in \ZZips$, we put  
\[
T_{\bs} \ := \ T^{1/a} \, \prod_{j=1}^\infty (1 + p s_j)^{q/a}\, \exp [-(\br,\bs)/a].
\]
In a completely similar way to the proof of Theorem \ref{theorem[E<]}, we can show that
\begin{equation}
|E^*(T)| 
\ \le \
\left(\frac{3}{2}\right)^{2m} \, T^{m/a}  \, 
\prod_{j=1}^\infty \left[1 + \sum_{s_j = 1}^\infty (1 + p s_j)^{qm/a}\, \exp (- m r_j s_j/a) \right]. \\[2ex]
\end{equation}
By use of Lemma  \ref{lemma[<sum_exp]} with $\eta:= m r_j/a$ and $q:=mq/a$, and the inequality \eqref{condition[P]}, we continue the estimation as
\begin{equation}
\begin{split}
|E^*(T)| 
\ &\le \
\left(\frac{3}{2}\right)^{2m} \, T^{m/a}  \, 
\prod_{j=1}^\infty \left[1 + 
\frac{(1 + p/2)^{qm/a} }{m r_j/a - pqm/a} \,  e^{- m r_j/2a}\right] \\[2ex]
\ &\le \
\left(\frac{3}{2}\right)^{2m} \, T^{m/a}  \, 
\exp \left[(1 + p/2)^{qm/a} \sum_{j=1}^\infty \, \frac{ e^{- m r_j/2a}}{m r_j/a - pqm/a}\right] \\[2ex]
\end{split}
\end{equation}
which gives the desired result.
\hfill\end{proof}

Note at this point that the assumption $m>0$ is crucial. For the  case $m=0$, our theorem would no longer give a rate in terms of $T$ but merely a constant.
Therefore, we shortly consider the case $m=0$ with $p=0$ in more detail. To this end, we  first focus on the $d$-dimensional case, i.e. we consider
$$
E_d(T) 
:= \ 
\big\{\bs \in \mathbb{Z}^d_+: \exp((\br,\bs)) \leq T\big\}.$$
We have the following bound for the cardinality $ |E_d(T)|$:
\begin{lemma} \label{lemma[F_+<]}
For $T \geq \exp(r_j),j=1,\ldots,d$, it holds
\begin{equation} \label{[F_+<](1)}
\left( \prod_{j=1}^d \frac1 r_j \right) \frac{\left( \log T \right)^d}{d!} 
\ \le \
|E_d(T)|
\ \le \
 \left( \prod_{j=1}^d \frac1 r_j \right) \frac{\left( \log T +\sum_{j=1}^d r_j \right)^d}{d!}.
\end{equation}
\end{lemma}
\begin{proof} 
We have
$$\log \left(
\prod_{j=1}^d e^{r_js_j} \right)= \sum_{j=1}^d r_js_j \leq {\log T}:=\tilde T$$  
which describes just an index set which is a 
weighted discrete simplex in $d$ dimensions of size $\Tilde T$.
Then we obtain with \cite{Dov} the desired result. For further aspects and a slight improvement on the lower bound, see also \cite{Padberg}.
\hfill\end{proof} 

This directly shows that we now only have a logarithmic cost rate in $T$, which is however exponential in the dimension $d$. 
Moreover, the order constant depends on $d$. 
Thus, we can not directly go to the infinite-dimensional limit $d\to \infty$.  
If we settle for a slightly higher cost rate instead, then, in certain cases, upper bounds can still be derived, which are completely independent of $d$ and also hold in the 
infinite-dimensional case.
To this end, note that we need to resolve for a given $T$ only up to a certain active dimension $d=d(T)$ since in the higher dimensions $d+j, j>0,$ we encounter just index values $s_{d+j}=0$ due to the definition of $E(T)$. Therefore, it holds 
$r_j \leq \log T, j=1,\ldots d$. Then, we get 
$$
\sum_{j=1}^{d} r_j  \leq  d \, {\log T}.
$$
Thus, with Lemma \ref{lemma[F_+<]}, using the lower bound of Stirling's formula $d! \geq \sqrt{2\pi d } (d/e)^d$ and 
$((d+1)/d)^d \leq e$, we obtain 
\begin{equation}\label{estlogT}
\begin{split}
|E_d(T)| 
&\leq   \prod_{j=1}^d \frac 1 {r_j}  \frac{(d+1)^d (\log T)^d}{d!} 
\leq   \prod_{j=1}^d \frac 1 {r_j}  \frac {e^{d}} {\sqrt{2 \pi d}} \frac{(d+1)^d} {d^d} (\log T)^d 
\leq  \prod_{j=1}^d \frac 1 {r_j}  \frac {e^{d+1}}{\sqrt{2 \pi d}}  (\log T)^d. 
\end{split}
\end{equation}

Now we consider the asymptotic behavior of the $r_j$ in more detail. 
For example, in the situation $r_j \geq e \cdot j$, we have $\prod_j^{d} 1/r_j \leq e^{-d} / d !$
and get
$$
\prod_{j=1}^d \frac1 {r_j} \frac {e^{d+1}}{\sqrt{2 \pi d}}  (\log T)^{d}
\leq 
\frac  {e^{-d}}{ d!} \frac {e^{d+1}}{\sqrt{2 \pi d}}  (\log T)^{d} =
 \frac {e}{\sqrt{2 \pi d}}  \frac {(\log T)^{d}}{d!} 
\leq  
 \frac {e}{\sqrt{2 \pi d}} {{\exp(\log T)}} \leq \frac 1 {\sqrt d} \frac e {2 \pi} T \leq \frac e {2 \pi} T ,
$$
i.e. we obtain a cost rate bound of $\frac e {2 \pi} T$. 
In the situation $r_j \geq e \cdot j^2$ we get analogously the bound
$$
\prod_{j=1}^d \frac 1 {r_j} \frac {e^{d+1}}{\sqrt{2 \pi d}}  (\log T)^{d}
 \leq
  \frac {e^{-d}}  {(d!)^2} \frac {e^{d+1}}{\sqrt{2 \pi d}}  (\log T)^{d} 
=
\frac {e}{\sqrt{2 \pi d}}  \left( \frac{\sqrt{\log T}^{d} } {d!}\right)^2
\leq 
\frac {e}{\sqrt{2 \pi d}} {e^{(2 \sqrt{\log T})}} 
\leq 
\frac {e}{\sqrt{2 \pi }} {e^{(2 \sqrt{\log T})}},
$$
which grows slower than any polynomial in $T$. If necessary, this bound can be further improved by a factor of $1/\sqrt{\log T}$.

In the general case, we have the following lemma:
\begin{lemma} \label{analyticinf}
Let $m=0$ and $p=0$. Furthermore, let $r_j \geq  \omega j^\tau$, where $\omega, \tau > 0$, and let
$$
c_{\omega,\tau}:= \tau + \omega^{-{1}/{\tau}} \log \left( \frac{e}{\omega} \right) > 0.
$$
Then, for the cardinality of the hyperbolic cross
$E(T)$ in the infinite-dimensional case,  there holds the bound
\begin{equation} 
|E(T)|
\ \le \ \exp\left( c_{\omega, \tau} \, {(\log T)}^{{1}/{\tau}} \right).
\end{equation}
\end{lemma}
\begin{proof} 
Consider again the finite-dimensional setting from Lemma \ref{lemma[F_+<]}. 
From the assumption  $r_j \geq \omega j^\tau$
it follows that
$$\prod_{j=1}^d \frac{1}{r_j} \leq \prod_{j=1}^d \omega^{-1} j^{-\tau} = \omega^{-d} (d!)^{-\tau},$$
where $d=d(T)$ again denotes the active dimension of $E(T)$.
Furthermore, since
$r_d \leq {\log T}$ it holds $d \leq \left( \frac{{\log T}}{\omega} \right)^{{1}/{\tau}}$.
Now, recall the upper bound from (\ref{estlogT}).
By employing $\frac{x^d}{d!} \leq \exp(x)$ we obtain for $d>1$, i.e. for $\log(T) > r_1$, the estimate
\begin{align*}
	|E_d(T)| & \leq	\prod_{j=1}^d \frac{1}{r_j} \frac{e^{d+1}}{\sqrt{2\pi d}} (\log T)^d   \leq \frac{e^{d+1}}{\sqrt{2\pi d} \omega^d (d!)^\tau} (\log T)^d 
		     = \frac{e^{d+1}}{ \omega^d \sqrt{2\pi d}} \left(  \frac{(\log T)^{{d}/{\tau} } }{d!} \right)^\tau \\
		     &\leq \frac{e^{d+1}}{ \omega^d \sqrt{2\pi d}} \exp\left( \tau (\log T)^{{1}/{\tau}} \right) 
		     = \frac{e}{\sqrt{2\pi d}} \exp\left( \tau (\log T)^{{1}/{\tau}}  + d \log\left( \frac{e}{\omega} \right) \right) \\
		    & \leq \frac{e}{\sqrt{2\pi d}} \exp\left( \tau (\log T)^{{1}/{\tau}}  +  \left( \frac{\log T}{\omega} \right)^{{1}/{\tau}} \log \left( \frac{e}{\omega} \right) \right) \\
		    & =  \frac{e}{\sqrt{2\pi d}} \exp\left[ (\log T)^{{1}/{\tau}} \left( \tau + \omega^{-{1}/{\tau}} \log \left( \frac{e}{\omega} \right) \right) \right]\\
		    & \leq \exp\left( c_{\omega, \tau} \, (\log T)^{{1}/{\tau}} \right) .
\end{align*}
Therefore, this is also a bound on $|E(T)|$.
\hfill\end{proof}

Altogether this shows a cost rate which is for $\tau>1$ better than any algebraic rate and only slightly worse than any logarithmic rate.\footnote{Or, 
vice versa, with cost $M = |E(T)|$ and thus $T \geq \exp \left( c_{\omega, \tau}^{-\tau} {(\log M)}^\tau \right) $,
 this gives a bound $\varepsilon \leq \exp(- c_{\omega,\tau}^{-\tau} {(\log M)}^\tau) $ for the accuracy $\varepsilon= T^{-1}$ that is obtained for fixed cost $M$.
Thus, we obtain a rate for the approximation error which is for $\tau>1$ better than any algebraic rate and is only slightly worse than any exponential rate. For  $0< \tau<1$, we get a meaningful rate at all.}
Moreover, for $0< \tau<1$, we get a meaningful estimate from our approach at all.
This has to be compared to results which were derived elsewhere by means of Stechkin's Lemma using so called $p$-summability and $\delta$-admissibility conditions that are based on complex analytic regularity assumptions on polydiscs/polyellipses.
In contrast to that, we only have to count the cardinality of the respective hyperbolic cross induced by $T$ and the smoothness indices $\br$. But we will have to use a Hilbert space structure (and not just a Banach space). Moreover, we have to assume a monotone ordering of the smoothness indices $\br$ to be given in the first place.

\subsection{Extensions} \label{Extensions}

In the previous two subsections, we gave upper bounds for the cardinality of hyperbolic crosses in the infinite-dimensional case
for the setting of nonperiodic mixed Sobolev-Korobov-type and nonperiodic Sobolev-analytic-type smoothness indices, i.e. we specifically considered the subsets 
$G(T) \subset \ZZmp \times \ZZips$ and $E(T) \subset  \ZZmp \times \ZZips$ from (\ref{G(T)}) and (\ref{E(T)}), respectively, that involve only nonnegative values in $\bk$ and $\bs$ which reflects the nonperiodic situation. Of course, periodic situations can be dealt with in an analogous way which would involve the sets  $ \ZZm$ and $\ZZis$ instead. This will be discussed in this subsection. 
Here, altogether, four different combinations of $ \ZZmp$ or $\ZZm$ with  $\ZZips$ or $\ZZis$ are possible. This leads, in addition to the 
fully nonperiodic case which we already dealt with, to corresponding results for the fully periodic case and for the nonperiodic-periodic and periodic-nonperiodic counterparts.  

To simplify the presentation, we need some new notation. We use the letter $\Ii$ to denote either 
$\ZZmp$ or $\ZZm$, and  $\Jj$ to denote either $\ZZips$ or $\ZZis$. Furthermore, we use the letter $S$ to denote either $K$ or $A$, i.e. the Korobov or the analytic setting.
We then can define the hyperbolic crosses in the infinite-dimensional case
\begin{equation} \label{G(T)-ext}
G_{\Ii \times \Jj}^S(T) 
:= \ 
\big\{(\bk,\bs)  \in \Ii \times \Jj: \, \lambda_S(\bk,\bs)   \leq T\big\} \quad \mbox{where} \quad S \in \left\{ K,A\right\}.
\end{equation}
Here, the scalars $\lambda_K(\bk,\bs) := \lambda(\bk,\bs)$ and $\lambda_A(\bk,\bs):= \rho(\bk,\bs)$ are defined in 
 \eqref{[lambda_{br}]} and \eqref{[lambda{a,mu,br}]}, respectively, with given and fixed corresponding parameters $m,a,t,r,\br,p,q$.
Furthermore, we denote by $\Jjd$  the set of all $\bs \in \Jj$ such that 
$\supp(\bs) \subset \{1,...,d\}$. With this notation we can identify  $\ZZd$ with $\Jjd$ for $\Jj = \ZZis$, 
and $\ZZdp$ with $\Jjd$ for $\Jj = \ZZips$. 
Then, we can analogously define  the hyperbolic crosses $G_{\Ii \times \Jjd}^K(T)$ and $G_{\Ii \times \Jjd}^A(T)$ in the finite-dimensional case by replacing $\Jj$ with $\Jjd$ in (\ref{G(T)-ext}), respectively.
We altogether obtain eight types of hyperbolic crosses for the infinite-dimensional case and eight types of  hyperbolic crosses for the $m+d$-dimensional case which correspond to the different possible combinations. 
For example, for the index sets of nonperiodic mixed Korobov-type and nonperiodic analytic smoothness, we now have 
$G_{\ZZmp \times \ZZip}^K(T) = G(T)$ and $E_{\ZZmp \times \ZZip}^A(T) = E(T)$ with $G(T)$ and $E(T)$ from \eqref{G(T)} and \eqref{E(T)}, respectively. 

We now can formulate the results on the estimation of $|G_{\Ii \times \Jj}^K(T)|$ and  $|G_{\Ii \times \Jj}^A(T)|$. 
To this end,  let us define the following quantities which depend on the parameters $m,a,t,r,\br,p,q$ and $T$:
\begin{equation}\label{[A^S(T)]}
A^S(T) := \
\begin{cases}
A(a,r,m,t,T), \  & \ S = K, \\[1ex]
T^{m/a}, \  & \ S = A, \\[2ex]
\end{cases}
\quad \quad \quad \mbox{ and } \quad \quad \quad
M^S  := \
\begin{cases}
M(t), \  & \ S = K, \\[1ex]
M_{p,q}, \  & \ S = A,
 \end{cases}
\end{equation}
with $A(a,r,m,t,T)$ from \eqref{AAA}, $M(t)$ from \eqref{condition[M]} and $M_{p,q}$ from \eqref{M00} for $p=0$ and from \eqref{Mpq} for $p>0$. Furthermore, let us define
\begin{equation}\label{[C_{Ii otimes Jj}^S]}
C_{\Ii \times \Jj}^S 
 := \
\begin{cases}
\exp{(M^S)}\, B(a,r,m,t), \  & \ S = K, \ \Ii = \ZZmp, \ \Jj = \ZZips, \\[1ex]
2^m\, \exp{(M^S)}\, B(a,r,m,t), \  & \ S = K, \ \Ii = \ZZm, \ \Jj = \ZZips, \\[1ex]
\exp{(2M^S)}\, B(a,r,m,t), \  & \ S = K, \ \Ii = \ZZmp, \ \Jj = \ZZis, \\[1ex]
2^m\, \exp{(2M^S)}\, B(a,r,m,t), \  & \ S = K, \ \Ii = \ZZm, \ \Jj = \ZZis, \\[1ex]
(3/2)^{2m}\exp{(M^S)}, \  & \ S = A, \ \Ii = \ZZmp, \ \Jj = \ZZips, \\[1ex]
2^m\,(3/2)^{2m}\exp{(M^S)}, \  & \ S = A, \ \Ii = \ZZm, \ \Jj = \ZZips, \\[1ex]
(3/2)^{2m}\exp{(2M^S)}, \  & \ S = A, \ \Ii = \ZZmp, \ \Jj = \ZZis, \\[1ex]
2^m\,(3/2)^{2m}\exp{(2M^S)}, \  & \ S = A, \ \Ii = \ZZm, \ \Jj = \ZZis, \\[1ex]
 \end{cases}
\end{equation}
with $ B(a,r,m,t)$ from \eqref{BBB}.

We are now in the position to state a generalized theorem which covers all different possible cases.
\begin{theorem} \label{theorem[Extension]}
For the case $S = K$, let the triple $m,a,\br$ be given as in \eqref{m,a,br[K]}, and,  for the case $S = A$, let the $5$-tuple $m,a,\br,p,q$ be given as in \eqref{m,a,br,p,q[A]}.
Suppose  that there hold the assumptions of Theorem  \ref{theorem[G<]} if $S = K$,  and the assumptions of 
Theorem \ref{theorem[E<]} if $S = A$, $p=0$, and of Theorem \ref{theorem[E<](2)} if $S = A$, $p>0$.
Then, we have for every $T \ge 1$
\begin{equation} \label{[G^S_IJ]}
|G_{\Ii \times \Jjd}^S(T)| 
\ \le \
 |G_{\Ii \times \Jj}^S(T)| 
\ \le \
C_{\Ii \times \Jj}^S \, A^S(T). 
\end{equation}
\end{theorem}
We here refrain from giving explicit proofs since they are similar to these of the last two subsections involving just obvious modifications.

We finish this section by giving a simple lower bound for $|G_{\Ii \times \Jjd}^S(T)|$ and $|G_{\Ii \times \Jj}^S(T)|$.

\begin{theorem} \label{theorem[lower-bound-G^S_IJ]}
For the case $S = K$, let the triple $m,a,\br$ be given as in \eqref{m,a,br[K]}, and,  for the case $S = A$, let the $5$-tuple $m,a,\br,p,q$ be given as in \eqref{m,a,br,p,q[A]}.
Then, we have for every $T \ge 1$
\begin{equation} \label{[lower-bound-G^S_IJ]}
 |G_{\Ii \times \Jj}^S(T)| 
\ \ge \
 |G_{\Ii \times \Jjd}^S(T)| 
\ \ge \
\lfloor T^{1/a} \rfloor^m.
\end{equation}
\end{theorem}
\begin{proof} It is sufficient to treat the case $\Ii = \ZZmp$, $\Jjd=\ZZdp$. 
To this end, we consider the set $Q$ of all elements $(\bk, {\bf0}) \in \ZZmp \times \ZZdp$ such that $ k_j \le T^{1/a} - 1$, $j =1,...,m$. 
Then, $Q$ is a subset of $G_{\ZZmp \times \ZZdp}^S(T)$ for $T \ge 1$. Hence, the obvious inequality 
$|Q| \ge \lfloor T^{1/a} \rfloor^m$ proves the assertion.
\hfill\end{proof}

\section{Hyperbolic cross approximation in infinite tensor product Hilbert spaces}  
\label{Approximation in Hilbert spaces}

In this section, we present the basic framework for our approximation theory in the infinite-dimensional case which is based on infinite tensor product Hilbert spaces. 
Then, we consider specific instances of our general theory and define the mixed Sobolev-Korobov-type and mixed Sobolev-analytic-type spaces mentioned in the introduction.

\subsection{Infinite tensor product Hilbert spaces and general approximation results}
We recall the notion of the infinite tensor product of separable Hilbert spaces. 
Let $H_j$, $j=1,...,m$, be separable Hilbert spaces with inner products $\langle \cdot,\cdot\rangle_j$. 
First, we define the finite-dimensional tensor product of $H_j$, $j=1,...,m$, as the tensor vector space 
$H_1 \otimes H_2 \otimes \cdots \otimes H_m$ equipped with the inner product
\begin{equation} \label{tensorporoduct-innerporducts}
 \langle\otimes_{j=1}^m\phi_j,\otimes_{j=1}^m\psi_j\rangle 
:= \ 
\prod_{j=1}^m\langle\phi_j,\psi_j\rangle_j \, \quad 
\mbox{for all} \ \phi_j,\psi_j \in H_j. 
\end{equation}
By taking the completion under this inner product, the resulting Hilbert space is defined as the tensor product space $H_1 \otimes H_2 \otimes \cdots \otimes H_m$  of $H_j$, $j=1,...,m$.
Next, we consider the infinite-dimensional case.
If $H_j, j \in \NN$, is a collection of separable Hilbert spaces and $\xi_j, j \in \NN$, is a collection of unit vectors in these Hilbert spaces then the infinite tensor product $\otimes_{j \in \NN} H_j$  is the completion of the set of all finite linear combinations of simple tensor vectors $\otimes_{j \in \NN} \phi_j$ where all but finitely many of the $\phi_j$'s are equal to the corresponding $\xi_j$. 
The inner product of 
$\otimes_{j \in \NN} \phi_j$ and $\otimes_{j \in \NN} \psi_j$ is defined as in \eqref{tensorporoduct-innerporducts}.
For details on infinite tensor product of Hilbert spaces, see \cite{BR02}. 

Now, we will need a tensor product of Hilbert spaces of a special structure. 
Let $H_1$ and $H_2$ be two given infinite-dimensional separable Hilbert spaces. Consider the infinite tensor product Hilbert space
\begin{equation} \label{H1H2}
 \Ll 
:= \ 
H_1^m \otimes H_2^\infty 
\quad \quad \mbox{ where } \quad \quad
H_1^m
:= \
\otimes_{j=1}^m H_1, \quad  
H_2^\infty
:= \ \otimes_{j=1}^\infty H_2. 
\end{equation}
In the following, we use the letters $I,J$ to denote either 
$\ZZ_+$ or $\ZZ$.  Recall also that we use the letter $\Ii$ to denote either $\ZZmp$ or $\ZZm$ and the letter $\Jj$ to denote either $\ZZips$  or $\ZZis$.
Let $\{\phi_{1,k}\}_{k \in I}$ and $\{\phi_{2,s}\}_{s \in J}$ be given orthonormal bases of $H_1$ and $H_2$, respectively. 
Then,  $\{\phi_{1,\bk}\}_{\bk \in \Ii}$ and $\{\phi_{2,\bs}\}_{\bs \in \Jj}$ are orthonormal bases of $H_1^m$ and $H_2^\infty$, respectively, 
where
\begin{equation} \label{ozzzzz}
 \phi_{1,\bk} 
:= \ 
 \otimes_{j=1}^m \phi_{1,k_j}, \quad
\phi_{2,\bs} 
:= \ 
\otimes_{j=1}^\infty \phi_{2,s_j}.
\end{equation}
Moreover, the set $\{\phi_{\bk,\bs}\}_{(\bk,\bs) \in \Ii \times \Jj}$ is an orthonormal basis of $\Ll$, where
\begin{equation} \label{oxxxxx}
 \phi_{\bk,\bs} 
:= \ 
 \phi_{1,\bk}  \otimes \phi_{2,\bs}.
\end{equation}
Thus, every $f \in \Ll$ can by represented by the series
\begin{equation} \label{}
f
\ = \ 
\sum_{(\bk,\bs) \in \Ii \times \Jj} f_{\bk,\bs}\,  \phi_{\bk,\bs},
\end{equation}
where $f_{\bk,\bs}:=\langle f,\phi_{\bk,\bs} \rangle$ is the $(\bk,\bs)$th coefficient of $f$ with respect to the orthonormal basis $\{\phi_{\bk,\bs}\}_{(\bk,\bs) \in \Ii \times \Jj}$. 
Furthermore, there holds Parseval's identity
\begin{equation} \label{[ParsevalId]}
\|f\|_\Ll^2 
\ = \ 
\sum_{(\bk,\bs) \in \Ii \times \Jj} |f_{\bk,\bs}|^2.
\end{equation}

Now let us assume that a general sequence of scalars 
$\lambda := \{\lambda(\bk,\bs)\}_{(\bk,\bs) \in \Ii \times \Jj}$ with 
$\lambda(\bk,\bs) \not= 0$ is given.
Then, we define the associated space $\Ll^\lambda$ as the set of all elements $f \in \Ll$ such that there exists a
$g \in \Ll$ such that
\begin{equation} \label{def[g-presentation]}
f := \sum_{(\bk,\bs) \in \Ii \times \Jj} \frac{g_{\bk,\bs}}{\lambda(\bk,\bs)} \, \phi_{\bk,\bs}.
\end{equation}
The norm of $\Ll^\lambda$ is defined by
$\|f\|_{\Ll^\lambda} \ := \|g\|_\Ll.$
From the definition \eqref{def[g-presentation]} and Parseval's identity we can see that 
\begin{equation} \label{P-Id[norm-Ll^lambda]}
\|f\|_{\Ll^\lambda}^2 
\ = \ 
\sum_{(\bk,\bs) \in \Ii \times \Jj} |\lambda(\bk,\bs)|^2 \,  |f_{\bk,\bs}|^2 .
\end{equation}
We also consider the subspace $\Ll_d$ in $\Ll$ defined by 
$\Ll_d
:= 
\left\{f 
 =  
\sum_{(\bk,\bs) \in \Ii \times \Jj_d} f_{\bk,\bs}\,  \phi_{\bk,\bs} \right\}
$
and the subspace
$\Ll^\lambda_d
:= \
\Ll^\lambda \cap \Ll_d,
$
where $\Jj_d:= \{\, \bs \in \Jj: \ \supp(\bs) \subset \{1, \cdots, d\}\, \}$.

Next, let us assume that the general nonzero sequences of scalars 
$\lambda := \{\lambda(\bk,\bs)\}_{(\bk,\bs) \in \Ii \times \Jj}$ and 
$\nu := \{\nu(\bk,\bs)\}_{(\bk,\bs) \in \Ii \times \Jj}$ are given with associated spaces $\Ll^\lambda$ and $\Ll^\nu$ with corresponding norms and subspaces $\Ll^\lambda_d$ and $\Ll^\nu_d$.
For $T \ge 0$, we define the set 
\begin{equation} \label{G(T)-H}
G_{\Ii \times \Jj}(T) 
:= \ 
\big\{(\bk,\bs)  \in \Ii \times \Jj: \, \frac{\lambda(\bk,\bs)}{\nu(\bk,\bs)}   \leq T\big\},
\end{equation}
and  the subspace $\Pp(T)$ of all $g \in \Ll$ of the form 
\begin{equation}
g
\ = \ 
\sum_{(\bk,\bs) \in G_{\Ii \times \Jj}(T)}  \, g_{\bk,\bs}\,  \phi_{\bk,\bs}.
\end{equation}
We are interested in the $\Ll^\nu$-norm approximation of elements from $\Ll^\lambda$  by elements from $\Pp(T)$. 
To this end, for $f \in \Ll$ and $T \ge 0$, we define the  operator $\Ss_T$ as
\begin{equation}
\Ss_T(f)
:= \
\sum_{(\bk,\bs) \in G_{\Ii \times \Jj}(T)}  \, f_{\bk,\bs}\,  \phi_{\bk,\bs}.
\end{equation}
We make the assumption throughout this section that $G_{\Ii \times \Jj}(T)$ is a finite set for every $T >0$.
Obviously, $\Ss_T$ is the orthogonal projection onto $\Pp(T)$. 
Furthermore, we define the set $G_{\Ii \times \Jj_d}(T)$, the subspace $\Pp_d(T)$
and the operator $\Ss_{d,T}(f)$ in the same way by replacing $\Jj$ by $\Jj_d$.

The following lemma gives an upper bound for the error of the orthogonal projection $\Ss_T$ with respect to the parameter $T$.
\begin{lemma} \label{lemma[Jackson]}
For arbitrary $T \geq 1$, we have
\begin{equation}
\|f - \Ss_T(f)\|_{\Ll^\nu}
\ \le  \
T^{-1}\|f\|_{\Ll^\lambda} \, , \qquad  \forall f \in \Ll^\lambda \cap  \Ll^\nu.
\end{equation}
\end{lemma}
\begin{proof}
Let $f \in \Ll^\lambda \cap  \Ll^\nu$.  From  
 the definition of the spaces $\Ll^\lambda$ and $\Ll^\nu$ and the definition \eqref{P-Id[norm-Ll^lambda]} of the associated norms $\|.\|_{\Ll^\lambda}$ and $\|.\|_{\Ll^\nu}$, we get
 \begin{equation*} 
\begin{aligned}
\|f - \Ss_T(f)\|_{\Ll^\nu}^2
\ &= \ 
\sum_{(\bk,\bs)\not\in G_{\Ii \times \Jj}(T)} |\nu(\bk,\bs)|^2  |f_{\bk,\bs}|^2 \\
\ & \le \
\sup_{(\bk,\bs)\not\in G_{\Ii \times \Jj}(T)} \left|\frac{\lambda(\bk,\bs)}{\nu(\bk,\bs)}\right|^{-2}
\ 
\sum_{(\bk,\bs)\not\in G_{\Ii \times \Jj}(T)} 
|\lambda(\bk,\bs)|^2 \, |f_{\bk,\bs}|^2
 \\
\ & \le \
T^{-2} \|f\|_{\Ll^\lambda}^2.
\end{aligned}
\end{equation*}  \hfill\end{proof}

Now  denote by $\Uu^\lambda$ the unit ball in $\Ll^\lambda$, i.e., 
$\Uu^\lambda := \ \{f \in \Ll^\lambda: \|f\|_{\Ll^\lambda} \le 1\},
$
and denote by $\Uu^\lambda_d$ the unit ball in $\Ll^\lambda_d$, i.e., 
$
\Uu^\lambda_d := \ \{f \in \Ll^\lambda_d: \|f\|_{\Ll^\lambda_d} \le 1\}.
$ We then have the following corollary:
\begin{corollary} \label{corollary[|f - S_T(f)|]}
For arbitrary $T \geq 1$, 
\begin{equation}
\sup_{f \in \Uu^\lambda} \ \inf_{g \in \Pp(T)} \|f - g\|_{\Ll^\nu}
\ =  \
\sup_{f \in \Uu^\lambda} \|f - \Ss_T(f)\|_{\Ll^\nu}
\ \le  \
T^{-1}. 
\end{equation}
\end{corollary}

Next, we give a Bernstein-type inequality. 
\begin{lemma} \label{lemna[Bernstein]}
For arbitrary $T \ge 1$, we have 
\begin{equation}
\|f\|_{\Ll^\lambda}
\ \le  \
T \|f\|_{\Ll^\nu}, \quad  \forall f \in \Pp(T).
\end{equation}
\end{lemma}
\begin{proof}
Let $f \in \Pp(T)$. From the definition of $\Ll^\lambda$ and $\Ll^\nu$ and the definition \eqref{P-Id[norm-Ll^lambda]} of the associated norms $\|.\|_{\Ll^\lambda}$ and $\|.\|_{\Ll^\nu}$, it follows that 
\begin{equation*} 
\begin{aligned}
\|f\|_{\Ll^\lambda}^2
\ &= \ 
\sum_{(\bk,\bs) \in  G_{\Ii \times \Jj}(T)} \,
|\lambda(\bk,\bs)|^2  |f_{\bk,\bs}|^2 \\[1.5ex]
\ & \le \
\sup_{(\bk,\bs)\in G_{\Ii \times \Jj}(T)} \left|\frac{\lambda(\bk,\bs)}{\nu(\bk,\bs)}\right|^2 \ 
\sum_{(\bk,\bs) \in G_{\Ii \times \Jj}(T)} |\nu(\bk,\bs)|^2 |f_{\bk,\bs}|^2  \\[1.5ex]
\ & \le \
T^{2}\|f\|_{\Ll^\nu}^2.
\end{aligned}
\end{equation*}\hfill\end{proof}

Now we are in the position to give lower and upper bounds on the $\varepsilon$-dimension $n_\varepsilon(\Uu^\lambda, \Ll^\nu)$.
\begin{lemma} \label{lemma[n_e]}
Let $\varepsilon \in (0,1]$. Then, we have
\begin{equation} \label{ineq[n_e]}
|G_{\Ii \times \Jj}(1/\varepsilon)| - 1
\ \le \
n_\varepsilon(\Uu^\lambda, \Ll^\nu)
\ \le  \
|G_{\Ii \times \Jj}(1/\varepsilon)|.
\end{equation}
\end{lemma}
\begin{proof} Put $T=1/\varepsilon$ and let 
\[
B(\varepsilon) := \{f \in  \Pp(T): \|f\|_{\Ll^\nu} \le \varepsilon\}.
\]
To prove the first inequality, we need the following result on Kolmogorov $n$-widths of the unit ball \cite[Theorem 1]{Ti60}:
Let $L_n$ be an $n$-dimensional subspace in a Banach space $X$, and let
$B_n(\delta):= \{f \in L_n: \ \|f\|_X \le \delta \}$, $\delta >0$. Then
\begin{equation} \label{[d_n(B,X)]}
d_{n-1}(B_n(\delta), X)
\ = \
\delta.
\end{equation}
In particular, for $n:= \dim \mathcal \Pp(T) = |G_{\Ii \times \Jj}(1/\varepsilon)|$, we get 
\begin{equation} \label{[d_{n-1}(B)]}
d_{n-1}(B(\varepsilon), \Ll^\nu)
\ = \ \varepsilon.
\end{equation}
Note furthermore that the definition \eqref{[d_n]} of the Kolmogorov $n$-width does not change if the outer infimum is taken over all linear manifolds $M_n$ in $X$ of dimension $n$ instead of dimension at most $n$. Hence, for every linear manifold 
$M_{n-1}$ in $ \Ll^\nu$ of dimension $n-1$,  equation \eqref{[d_{n-1}(B)]} yields 
\[
 \sup_{f \in B(\varepsilon)} \ \inf_{g \in M_{n-1}} \|f - g\|_{\Ll^\nu} \ge \varepsilon.
\] 
From Lemma \ref{lemna[Bernstein]}, we obtain
\begin{equation}  \label{[B(epsilon)]}
 B(\varepsilon) \ \subset \ \Uu^\lambda,
\end{equation}
which gives
\[
n_\varepsilon(\Uu^\lambda, \Ll^\nu)
\ \ge \ 
n_\varepsilon(B(\varepsilon), \Ll^\nu)
\ \ge \ 
\sup \left\{n':\, \forall  \ M_{n'}: 
\ \dim M_{n'} \le n', \ \sup_{f \in  B(\varepsilon)} \ \inf_{g \in M_{n'}} \|f - g\|_{\Ll^\nu} \ge \varepsilon \right\}.
\] 
Here, $M_{n'}$ is a linear manifold in $ \Ll^\nu$ of dimension $\le n'$.
Altogether, this
proves the first inequality in \eqref{ineq[n_e]}.
The second inequality follows from Corollary \ref{corollary[|f - S_T(f)|]}. 
\hfill\end{proof}

In a similar way, by using
the set $G_{\Ii \times \Jj_d}(T)$, the subspace $\Pp_d(T)$
and the operator $\Ss_{d,T}(f)$, we can prove the following lemma for $n_\varepsilon(\Uu^\lambda_d, \Ll^\nu_d)$. 
\begin{lemma} \label{lemma[n_e(d)]}
Let $\varepsilon \in (0,1]$. Then we have
\begin{equation} \label{ineq[n_e(d)]}
|G_{\Ii \times \Jj_d}(1/\varepsilon)| - 1
\ \le \
n_\varepsilon(\Uu^\lambda_d, \Ll^\nu_d)
\ \le  \
|G_{\Ii \times \Jj_d}(1/\varepsilon)|.
\end{equation}
\end{lemma}

\subsection{Results for mixed Sobolev-Korobov-type and mixed Sobolev-analytic-type spaces}

So far, we laid out the basic framework for our theory in the infinite-dimensional case, where the error of an approximation is measured 
in the $\|.\|_{\Ll^\nu}$-norm and the functions to be approximated are from the space $\Ll^\lambda$ with associated, given general sequences $\nu$ and $\lambda$.
In the following, we will get more specific and we will plug in particular sequences $\nu$ and $\lambda$. They define the 
Sobolev-Korobov-type spaces and Sobolev-analytic-type spaces mentioned in the introduction, whose indices were already used for the definition of the hyperbolic crosses in Section \ref{cardinality of HC}.
To this end, we will use Sobolev-type spaces for $H_1^m$, Korobov-type spaces for $H_2^\infty$, and analytic-type spaces for $H_2^\infty$ in (\ref{H1H2}). 

First, for given $\alpha \ge 0$, we define the scalar $\lambda_{m,\alpha}(\bk)$, $\bk \in \Ii$, by 
\begin{equation} \label{defa_index}
\lambda_{m,\alpha}(\bk) 
 := \
\max_{1\le j \le m}(1 + |k_j|)^\alpha. 
\end{equation}
Recall the definition (\ref{H1H2}). The Sobolev-type space $K^{\alpha}$ is then defined as the set of all functions $f \in H_1^m$ such that there exists a
$g \in H_1^m$ such that
\begin{equation} \label{def[K^m,alpha]}
f = \sum_{\bk \in \Ii}  \frac{g_\bk}{\lambda_{m,\alpha}(\bk)} \, \phi_{1,\bk},
\end{equation}
where $g_\bk := \langle g,\phi_{1,\bk} \rangle$ is the $\bk$th coefficient of $g$ with respect to the orthonormal basis $\{\phi_{1,\bk}\}_{\bk \in \Ii}$. 
The norm of $K^{\alpha}$ is defined by
\begin{equation}\label{norm[K^m,alpha]}
\|f\|_{K^{\alpha}} \ := \|g\|_\Ll.
\end{equation}

Second, for given $\br \in \RR^\infty_+$, we define the scalar  $\rho_\br(\bs)$, $\bs \in \Jj$, by
\begin{equation} \label{[rho^br]}
\rho_\br(\bs) 
 := \
\prod_{j=1}^\infty(1 + |s_j|)^{r_j}. 
\end{equation}
Recall the definition (\ref{H1H2}). The Korobov-type space $K^{\br}$ is then defined as the set of all functions $f \in H_2^\infty$ such that there exists a
$g \in H_2^\infty$ such that
\begin{equation} \label{def[Kbr]}
f = \sum_{\bs \in \Jj}  \frac{g_\bs}{\rho_\br(\bs)} \,  \phi_{2,\bs},
\end{equation}
where $g_\bs := \langle g,\phi_{2,\bs} \rangle$ is the $\bs$th coefficient of $g$ with respect to the orthonormal basis $\{\phi_{2,\bs}\}_{\bs \in \Jj}$.
The norm of $K^{\br}$ is defined by 
\begin{equation}\label{norm[Kbr]}
\|f\|_{K^{\br}} \ := \|g\|_\Ll.  
\end{equation}

Third, for given $\br \in \RR^\infty_+$, $p,q \ge 0$, we define the scalar  $\rho_{\br,p,q}(\bs)$, $\bs \in \Jj$, by
\begin{equation}\label{[rho^br,pq]}
\rho_{\br,p,q}(\bs)
 := \
\prod_{j=1}^\infty (1 + p|s_j|)^{-q}\, \exp \left((\br,|\bs|)\right), \quad 
(\br,|\bs|):= \sum_{j=1}^\infty r_j|s_j|.
\end{equation}
The analytic-type space $A^{\br,p,q}$ and its norm are then defined as in \eqref{def[Kbr]} and \eqref{norm[Kbr]} by replacing 
$\rho_\br(\bs)$ with $\rho_{\br,p,q}(\bs)$.
Note at this point that the spaces  $K^{\alpha}$, $K^{\br}$ and $A^{\br,p,q}$ are themselves Hilbert spaces with their naturally induced inner product. This means that if $g,g'$ represent $f,f'$ as in \eqref{norm[K^m,alpha]} or \eqref{norm[Kbr]}, then  $\langle f, f' \rangle:= \langle g, g' \rangle$.

For $\alpha \geq \beta \geq 0$, we now define the spaces $\Gg$ and $\Hh$ by
\begin{equation} \label{[GgHh]}
\Gg 
 := \
K^{\beta} \otimes H_2^\infty \quad  \quad \quad
\Hh 
 := \
K^{\alpha} \otimes F, \quad \mbox{where } F \ \mbox{ is either} \ K^{\br} \ \mbox{or} \ A^{\br,p,q}.
\end{equation}
Here, $ H_2^\infty$ is given in (\ref{H1H2}).
The space $\Hh$ is called Sobolev-Korobov-type space if $F=K^{\br}$  in \eqref{[GgHh]} and Sobolev-analytic-type space if 
$F=A^{\br,p,q}$.
From these definitions we can see that $\Gg$ and $\Hh$ are special cases of $\Ll $ in (\ref{H1H2}). Moreover, 
\begin{equation} \label{Gg}
\Gg 
\ = \ 
\Ll^\nu, \quad\mbox{where } \nu := \{\nu(\bk,\bs)\}_{(\bk,\bs) \in \Ii \times \Jj}, \quad \nu(\bk,\bs):= \lambda_{m,\beta}(\bk).
\end{equation}
Furthermore, denoting by $S$ either $K$ or $A$ (cf. Subsection \ref{Extensions}), we see that  
\begin{equation} \label{Hh}
\Hh 
\ = \ 
\Ll^\lambda, \quad\mbox{where } \lambda := \{\lambda(\bk,\bs)\}_{(\bk,\bs) \in \Ii \times \Jj}, \quad
\lambda(\bk,\bs):= \lambda_S(\bk,\bs),
\end{equation}
where $ \lambda_S(\bk,\bs)=\lambda_{m,\alpha}(\bk)\rho_\br(\bs)$ if $S=K$ and  $ \lambda_S(\bk,\bs)=\lambda_{m,\alpha}(\bk)\rho_{\br,p,q}(\bs)$ if $S=A$, respectively.
We also consider the subspaces 
$
\Gg_d
:= 
\Gg \cap \Ll_d$ and $ \Hh_d := \Hh \cap \Ll_d.
$

We are now in the position to state the properties of the hyperbolic cross approximation of functions from $\Hh$ with respect to the $\Gg$-norm. 
To this end, we fix the parameters $m,\alpha,\beta,\br,p,q$ in the definition of 
$\Gg$ and $\Hh$, and put $a=\alpha-\beta$. 
We will assume that, for $S = K$, the triple 
$m,a,\br$ is given as in \eqref{m,a,br[K]} and, for $S = A$, the $5$-tuple $m,a,\br,p,q$ is given as in \eqref{m,a,br,p,q[A]}.
Denote by $\Uu$ the unit ball in $\Hh$, i.e., 
$\Uu := \ \{f \in \Hh: \|f\|_\Hh \le 1\},
$
and by $\Uu_d$ the unit ball in $\Hh_d$, i.e., 
$
\Uu_d := \ \{f \in \Hh_d: \|f\|_\Hh \le 1\}.
$
Then, from  Corollary \ref{corollary[|f - S_T(f)|]},  Lemmata \ref{lemma[n_e]} and \ref{lemma[n_e(d)]} and Theorems \ref{theorem[Extension]} and \ref{theorem[lower-bound-G^S_IJ]}, we obtain the following result:
\begin{theorem} \label{theorem[n_e]}
Let $\alpha > \beta \ge 0$, let $a= \alpha - \beta$, let, for $S = K$, the triple $m,a,\br$ be given as in \eqref{m,a,br[K]}, and  let, for $S = A$, the $5$-tuple $m,a,\br,p,q$ be given as in \eqref{m,a,br,p,q[A]}.
Suppose  that there hold the assumptions of Theorem  \ref{theorem[G<]} if $S = K$,  and the assumptions of 
Theorem \ref{theorem[E<]} if $S = A$, $p=0$, and Theorem \ref{theorem[E<](2)} if $S = A$, $p>0$.
Then, we have for every $d \in \NN$ and every $\varepsilon \in (0,1]$
\begin{equation} \label{ineq[n_e<]}
\lfloor \varepsilon^{- 1/(\alpha - \beta)} \rfloor^m - 1
\ \le \
n_\varepsilon(\Uu_d, \Gg_d)
\ \le \
n_\varepsilon(\Uu, \Gg)
\ \le \
C_{\Ii \times \Jj}^S \, A^S(\varepsilon^{-1}), 
\end{equation}
where $C_{\Ii \times \Jj}^S$ and $A^S(T)$ are as in \eqref{[C_{Ii otimes Jj}^S]} and \eqref{[A^S(T)]}, respectively.
Moreover, it holds
\begin{equation} \label{approx}
\sup_{f \in \Uu_d} \|f - \Ss_{d,\varepsilon^{-1}}(f)\|_{\Gg_d}
\ \le  \
\sup_{f \in \Uu} \|f -\Ss_{\varepsilon^{-1}}(f)\|_\Gg
\ \le  \
\varepsilon. 
\end{equation}
\end{theorem}

\section{Hyperbolic cross approximation of specific infinite-variate functions: Two examples}  
\label{Function approximation}

In this section, we make the results of the previous section for hyperbolic cross approximation of infinite-variate functions more specific. 
We will consider two situations as examples: First, the approximation of infinite-variate periodic functions from Sobolev-Korobov-type spaces and, second,  the approximation of infinite-variate nonperiodic functions from Sobolev-analytic-type spaces. Note that the other possible cases can be treated in an analogous way.

\subsection{Approximation of infinite-variate periodic functions}

Denote by $\TT$ the one-dimensional torus represented as the interval $[0,1]$ with identification of the end points $0$ and $1$. Let us define a
probability measure on $\TTi$.  It is the
infinite tensor product measure $\mu$ of the univariate Lebesgue measures on the one-dimensional $\TT$, i.e.
$
\mathrm{d} \mu(\by) 
\ = \ 
\bigotimes_{j \in \NN} \mathrm{d} y_j.
$
Here, the sigma algebra $\Sigma$ for $\mu$ is generated by the periodic finite rectangles
$\prod_{j \in \ZZ} I_j
$
where only
a finite number of the $I_j$ are different from $\TT$ and those that are different are periodic intervals
contained in $\TT$. Then, $(\TTi, \Sigma, \mu )$ is a probability space.

Now, let $L_2(\TTi):= L_2(\TTi, \mu)$ denote the Hilbert space of functions on $\TTi$ equipped
with the inner product
$\langle f,g \rangle 
:= \
\int_{\TTi} f(\by) \overline{g(\by)} \, \mathrm{d} \mu(\by).
$
The norm in $L_2(\TTi)$ is defined as $\|f\| := \langle f,f \rangle^{1/2}$.
Furthermore, let $L_2(\TTm)$ be the usual Hilbert space of Lebesgue square-integrable functions on $\TTm$.
Then, we set
\begin{equation} \nonumber
L_2(\TTm \times \TTi)
:= \
L_2(\TTm) \otimes L_2(\TTi).
\end{equation}
Observe that this and other similar definitions become an equality if we consider the tensor product 
measure in $\TTm \times \TTi$.
For $(\bk,\bs)\in \ZZm \times \ZZis$, we define
\begin{equation} \nonumber
e_{(\bk,\bs)}(\bx,\by)
:= \ 
e_{\bk}(\bx)e_{\bs}(\by), 
\quad e_{\bk}(\bx):= \ \prod_{j=1}^m e_{k_j}(x_j), \
e_{\bs}(\by):= \ \prod_{j \in \supp(\bs)} e_{s_j}(y_j),
\end{equation}
where $e_{s}(y):=  e^{i2\pi s y}$. 
Note that $\{e_{(\bk,\bs)}\}_{(\bk,\bs) \in \ZZm \times \ZZis}$ is an orthonormal basis of 
$L_2(\TTm \times \TTi)$. Moreover, for every 
$f \in L_2(\TTm \times \TTi)$, we have the following expansion 
\begin{equation} \nonumber
f = \sum_{(\bk,\bs) \in  \ZZm \times \ZZis} \hat{f}(\bk,\bs)e_{(\bk,\bs)},
\end{equation}
where, for  $(\bk,\bs) \in \ZZm \times \ZZis$, 
$\hat{f}(\bk,\bs): = \langle f, e_{(\bk,\bs)} \rangle$  is the $(\bk,\bs)$th Fourier coefficient of $f$. Hence, putting 
$H_1 = H_2 = L_2(\TT)$, we have
\begin{equation} \nonumber
L_2(\TTm \times \TTi)
\ = \
\Ll 
:= \
H_1^m \otimes H_2^\infty.
\end{equation} 

Next, based on the orthonormal bases 
$\{\phi_{1,k}\}_{k \in I} := \{e_k\}_{k \in \ZZ}$ and  $ \{\phi_{2,s}\}_{s \in J} := \{e_s\}_{s \in \ZZ}$ for the two spaces $H_1$ and $H_2$ in (\ref{H1H2}) with 
$I=J=\ZZ$, respectively, we construct the associated Sobolev-type spaces $K^{\beta}(\TTm \times \TTi)$ 
and the associated Sobolev-Korobov-type spaces $K^{\alpha,\br}(\TTm \times \TTi)$ for the periodic case as
\[
K^{\beta}(\TTm \times \TTi):= \Gg, \quad K^{\alpha,\br}(\TTm \times \TTi):= \Hh,
\]
where $\Gg$ and $\Hh$ are defined\footnote{At this point, note a slight abuse of notation. In \eqref{[GgHh]}, the $K^\beta$ only relates to $H^m_1$. 
} 
as in \eqref{[GgHh]} with $F= K^\br$ and the triple 
$m,\alpha,\br$  as in \eqref{m,a,br[K]}. 
Furthermore, we set 
\[
K^{\beta}_d(\TTm \times \TTi):= \Gg_d, \quad K^{\alpha,\br}_d(\TTm \times \TTi):= \Hh_d.
\]
For $T \ge 0$, let us denote by $\Tt(T)$ the subspace of trigonometric polynomials $g$ of the form 
\begin{equation}
g
:= \
\sum_{(\bk,\bs) \in G_{\ZZm \times \ZZis}^K(T)}  \hat{g}(\bk,\bs)e_{(\bk,\bs)},
\end{equation}
where 
\begin{equation} \nonumber
G_{\ZZm \times \ZZis}^K(T) 
:= \ 
\big\{(\bk,\bs)  \in \ZZm \times \ZZis: \lambda_K(\bk,\bs)   \leq T\big\},
\end{equation}
and, with $a=\alpha-\beta$,  
\begin{equation} \nonumber
\lambda_K(\bk,\bs) 
 := \
\max_{1\le j \le m}(1 + |k_j|)^a\prod_{j=1}^{t+1}(1 + |s_j|)^r 
\prod_{j=t+2}^\infty(1 + |s_j|)^{r_j}. 
\end{equation}
For  $f \in L_2(\TTm \times \TTi)$ and $T \ge 0$, we define the  Fourier operator $\Ss_T$ as
\begin{equation}
\Ss_T(f)
:= \
\sum_{(\bk,\bs) \in G_{\ZZm \times \ZZis}^K(T)}  \hat{f}(\bk,\bs)e_{(\bk,\bs)}.
\end{equation}

Let $U^{\alpha,\br}(\TTm \times \TTi):= \Uu$ and $U_d^{\alpha,\br}(\TTm \times \TTi):= \Uu_d$ be the unit ball in 
$K^{\alpha,\br}(\TTm \times \TTi)$ and $K_d^{\alpha,\br}(\TTm \times \TTi)$, respectively. Now, from Theorems \ref{theorem[n_e]}, \ref{theorem[Extension]} and \ref{theorem[lower-bound-G^S_IJ]}, the results on the hyperbolic cross approximation in infinite tensor product Hilbert spaces of Section 
\ref{Approximation in Hilbert spaces} can be reformulated for the approximation of periodic functions in Sobolev-Korobov-type spaces as follows.

\begin{theorem} \label{theorem[n_e]-periodic}
Let $\alpha > \beta \ge 0$, let $a= \alpha - \beta$ and let the triple $m,a,\br$ be given as in \eqref{m,a,br[K]}.
Suppose  that there hold the assumptions of Theorem  \ref{theorem[G<]}. With
\begin{equation} \nonumber
n_\varepsilon(d)
:= \
{n_\varepsilon(U_d^{\alpha,\br}(\TTm \times \TTi), K_d^{\beta}(\TTm \times \TTi))}, \quad
n_\varepsilon
:= \
n_\varepsilon(U^{\alpha,\br}(\TTm \times \TTi), K^{\beta}(\TTm \times \TTi)),
\end{equation}
we have for every $d \in \NN$ and every $\varepsilon \in (0,1]$
\begin{equation} \label{ineq[n_e<]-periodic}
\lfloor \varepsilon^{- 1/(\alpha - \beta)} \rfloor^m - 1
\ \le \
n_\varepsilon(d)
\ \le \
n_\varepsilon
\ \le \
C \, A(\varepsilon^{-1}), 
\end{equation}
where $C := 2^m\, e^{2M(t)}\, B(\alpha - \beta,r,m,t)$ with $M(t)$ from \eqref{condition[M]} and $B(a,r,m,t)$ from \eqref{BBB}, and
\begin{equation}\nonumber
A(T) \ := \
\begin{cases}
T^{m/(\alpha - \beta)}, \  & \ r > (\alpha - \beta)/m, \\[2ex]
T^{m/(\alpha - \beta)} \log (2T^{1/(\alpha-\beta)} + 1)  [r^{-1}\log T + (t + 1) \log 2]^t , \  & \ r = (\alpha - \beta)/m,\\[2ex]
T^{1/r} [r^{-1}\log T + (t + 1) \log 2]^t , \  & \ r < (\alpha - \beta)/m.
 \end{cases}
\end{equation}
\end{theorem}

Note at this point that we discussed here only the example involving the Sobolev-Korobov-type space. 
The Sobolev-analytic-type space as well as other combinations for the periodic case can be defined and dealt with in an analogous way if necessary.

\subsection{Approximation of infinite-variate nonperiodic functions} 
\label{Nonperiodic approximation}
In the following, we consider the nonperiodic case in more detail. Here, we focus on two types of domains,  $\RR$ and $\II := [-1,1]$. 
To this end, we use the letter $\DD$ to denote either $\II$ or $\RR$. Let us define a 
probability measure $\mu$ on $\DDi$.  For $\DD= \II$, a probability measure on $\IIi$  is the
infinite tensor product measure $\mu$ of the univariate uniform probability measures on the one-dimensional $\II$, i.e.
$
\mathrm{d} \mu(\by) 
\ = \ 
\bigotimes_{j \in \ZZ} \frac{1}{2}\mathrm{d} y_j.
$
For $\DD= \RR$, a probability measure on $\RRi$ is the
infinite tensor product measure $\mu$ of the univariate Gaussian probability measure on $\RR$, i.e.
$
\mathrm{d} \mu(\by) 
\ = \ 
\bigotimes_{j \in \ZZ} (2 \pi)^{-1/2} \exp(-y_j^2/2)\mathrm{d} y_j.
$
Here, the sigma algebra $\Sigma$ for $\mu$ is generated by the finite rectangles
$ 
\prod_{j \in \NN} I_j,
$
where only
a finite number of the $I_j$ are different from $\DD$ and those that are different are intervals
contained in $\DD$. Then, $(\DDi, \Sigma,  \mu)$ is a probability space.

Now, let $L_2(\DDi, \mu)$ denote the Hilbert space of functions on $\DDi$ equipped
with the inner product
$
\langle f,g \rangle 
:= \
\int_{\DDi} f(\by) \overline{g(\by)} \, \mathrm{d} \mu(\by).
$
The norm in $L_2(\DDi, \mu)$ is defined as $\|f\| := \langle f,f \rangle^{1/2}$. In what follows, 
$\mu$ is fixed, and for convention, we write $L_2(\DDi, \mu):= L_2(\DDi)$. Furthermore, let $L_2(\IIm)$ be the usual Hilbert space of Lebesgue square-integrable functions on $\IIm$ based on the univariate 
normed Lebesgue measure. 
Then, we define 
\begin{equation} \nonumber
L_2(\IIm \times \DDi)
:= \
L_2(\IIm) \otimes L_2(\DDi).
\end{equation}

Let $\{l_k\}_{k=0}^\infty$ be the family of univariate orthonormal Legendre polynomials in 
$L_2(\II)$ and let $\{h_k\}_{k=0}^\infty$ be the family of univariate orthonormal Hermite polynomials in 
$L_2(\RR,  \mu)$ with associated univariate measure $\mathrm{d} \mu(y) := (2 \pi)^{-1/2} \exp(-y^2/2)\mathrm{d} y$. We set
\begin{equation} \label{L-polynomials}
\{\phi_{1,k}\}_{k \in \ZZ_+}
:= \ \{l_k\}_{k \in \ZZ_+},
\quad \mbox{and} \quad 
\{\phi_{2,s}\}_{s \in \ZZ_+} 
:= \ 
\begin{cases}
\{l_s\}_{s \in \ZZ_+}, \ & \DD = \II, \\[1ex]
\{h_s\}_{s \in \ZZ_+}, \ & \DD = \RR.
\end{cases} 
\end{equation}
For $(\bk,\bs)\in \ZZmp \times \ZZips$, we define
\begin{equation} \nonumber
\phi_{(\bk,\bs)}(\bx,\by)
:= \ 
\phi_{1,\bk}(\bx)\phi_{2,\bs}(\by), 
\quad \phi_{1,\bk}(\bx):= \ \prod_{j=1}^m \phi_{1,k_j}(x_j), \
\phi_{2,\bs}(\by):= \ \prod_{j \in \supp(\bs)} \phi_{2,s_j}(y_j).
\end{equation}
Note that $\{\phi_{(\bk,\bs)}\}_{(\bk,\bs) \in \ZZmp \times \ZZips}$ is an orthonormal basis of 
$L_2(\IIm \times \DDi)$. Moreover, for every 
$f \in L_2(\IIm \times \DDi)$, we have the following expansion 
\begin{equation} \nonumber
f = \sum_{(\bk,\bs) \in \ZZmp \times \ZZips} f_{\bk,\bs}\phi_{(\bk,\bs)},
\end{equation}
where for  $(\bk,\bs) \in \ZZmp \times \ZZips$, 
$f_{\bk,\bs}: = \langle f, \phi_{(\bk,\bs)} \rangle$  denotes the $(\bk,\bs)$th coefficient of $f$ with respect to the orthonormal basis $\{\phi_{(\bk,\bs)}\}_{(\bk,\bs) \in \ZZmp \times \ZZips}$. Hence, by putting 
$H_1= L_2(\II, \frac{1}{2} \mathrm{d}x)$ and $H_2 = L_2(\DD)$, we have
\begin{equation} \nonumber
L_2(\IIm \times \DDi)
\ = \
\Ll 
:= \
H_1^m \otimes H_2^\infty.
\end{equation}

Next, based on the orthonormal bases 
$\{\phi_{1,k}\}_{k \in I}$ and $\{\phi_{2,s}\}_{s \in J}$  for the two spaces $H_1$ and $H_2$ in (\ref{H1H2}) with 
$I= J=\ZZ_+$ as defined in \eqref{L-polynomials}, respectively, we construct the associated Sobolev-Korobov-type spaces $K^{\beta}(\IIm \times \DDi)$ and the associated Sobolev-analytic-type spaces 
$A^{\alpha,\br,p,q}(\IIm \times \DDi)$ for the nonperiodic case as
\[
K^{\beta}(\IIm \times \DDi):= \Gg, \quad A^{\alpha,\br,p,q}(\IIm \times \DDi):= \Hh,
\]
where $\Gg$ and $\Hh$ are defined\footnote{Note again a slight abuse of notation here. In \eqref{[GgHh]}, the $K^\beta$ only relates to $H^m_1$. 
}
as in \eqref{[GgHh]} with $F= A^{\br,p,q}$ and the $5$-tuple 
$m,\alpha,\br,p,q$ as in \eqref{m,a,br,p,q[A]}.
Furthermore, we set  
\[
K^{\beta}_d(\IIm \times \DDi):= \Gg_d, \quad A^{\alpha,\br,p,q}_d(\IIm \times \DDi):= \Hh_d.
\]

For $T \ge 0$, we  denote by $\Pp^A(T)$ the subspace of polynomials $g$ of the form 
\begin{equation} \label{Pp^A(T)}
g
:= \
\sum_{(\bk,\bs) \in G_{\ZZmp \times \ZZips}^A(T)}   g_{\bk,\bs} \phi_{(\bk,\bs)},
\end{equation}
where 
\begin{equation} \label{G_{ZZmp X ZZips}^A}
G_{\ZZmp \times \ZZips}^A(T) 
:= \ 
\big\{(\bk,\bs)  \in \ZZmp \times \ZZips: \lambda_A(\bk,\bs)   \leq T\big\},
\end{equation}
and 
\begin{equation} \label{def_lam_A}
\lambda_A(\bk,\bs) 
 := \
 \max_{1\le j \le m}(1 + k_j)^a \, \prod_{j=1}^\infty (1 + p s_j)^{-q}\, \exp ((\br,\bs)), \quad 
(\br,\bs):= \sum_{j=1}^\infty r_js_j.
\end{equation}

For $f \in L_2(\IIm \times \DDi)$ and $T \ge 0$, we define the operator $\Ss_T^A$ as
\begin{equation} \label{[S_T]-nonperiodic}
\Ss_T^A(f)
:= \
\sum_{(\bk,\bs) \in G_{\ZZmp \times \ZZips}^A(T)}  f_{\bk,\bs}\phi_{(\bk,\bs)}.
\end{equation}
Let $U^{\alpha,\br,p,q}(\IIm \times \DDi):= \Uu$ and $U_d^{\alpha,\br,p,q}(\IIm \times \DDi):= \Uu_d$ be the unit ball in $A^{\alpha,\br,p,q}(\IIm \times \DDi)$ and $A_d^{\alpha,\br,p,q}(\IIm \times \DDi)$, respectively.
Now, from the Theorems \ref{theorem[n_e]}, \ref{theorem[Extension]} and \ref{theorem[lower-bound-G^S_IJ]}, the results on hyperbolic cross approximation in infinite tensor product Hilbert spaces of Section 
\ref{Approximation in Hilbert spaces} can be reformulated for the approximation of nonperiodic functions in Sobolev-analytic-type spaces as follows.
\begin{theorem} \label{theorem[n_e]-nonperiodic}
Let $\alpha > \beta \ge 0$, let $a= \alpha - \beta$ and let the $5$-tuple 
$m,a,\br,p,q$ be given as in \eqref{m,a,br,p,q[A]}.
Suppose  that there hold the assumptions of Theorem \ref{theorem[E<]} if $p=0$, and the assumption of 
Theorem \ref{theorem[E<](2)} if  $p>0$. 
With
\begin{equation} \nonumber
n_\varepsilon(d)
:= \
n_\varepsilon {(U_d^{m,\alpha,\br,p,q}(\IIm \times \DDi), K^{\beta}_d(\IIm \times \DDi))}, \quad
n_\varepsilon
:= \
n_\varepsilon(U^{m,\alpha,\br,p,q}(\IIm \times \DDi), K^{\beta}(\IIm \times \DDi))
\end{equation}
we have for every $d \in \NN$ and every $\varepsilon \in (0,1]$
\begin{equation} \label{ineq[n_e<]-nonperiodic}
\lfloor \varepsilon^{- 1/(\alpha - \beta)} \rfloor^m - 1
\ \le \
n_\varepsilon(d)
\ \le \
n_\varepsilon
\ \le \
(3/2)^m \exp({M_{p,q}}) \, \varepsilon^{- m/(\alpha - \beta)}, 
\end{equation}
where $M_{p,q}$ is as in \eqref{M00} for $p=0$, and as in \eqref{Mpq} for $p>0$.
\end{theorem}

Note at this point that we discussed here only the example involving the Sobolev-analytic-type space.  
Other combinations can be defined and dealt with in an analogous way. 

\section{Application}\label{application}
We now give an example of the application of our approximation results in the field of uncertainty quantification. We focus on the notorious model problem
\begin{equation}\label{spde}
-\mathrm{div}_\bx(\sigma(\bx,\by)\nabla_\bx u(\bx,\by)) = f(\bx) \quad \bx \in \II^m \quad \by \in \DD^\infty,
\end{equation} 
with homogeneous boundary conditions $u(\bx,\by)=0$, $\bx\in \partial \II^m $, $\by \in \DD^\infty$,
i.e. we have to find a real-valued function $u: \II^m \times \DD^\infty \to \mathbb{R}$ such that (\ref{spde}) holds $ \mu$-almost everywhere, where $\DD^\infty$ is either $\IIi$ or $\RRi$ and $ \mu$ is the infinite tensor product probability measure on $\DD^\infty$ defined in Subsection \ref{Nonperiodic approximation}.
Here, $\II^m$ represents the domain of the physical space, which  is usually $m=1,2,3$-dimensional, and $\DD^\infty$ represents the infinite-dimensional stochastic or parametric domain. We assume that there holds the uniform ellipticity condition $0 < \sigma_{min} \leq \sigma (\bx,\by) \leq \sigma_{max}< \infty$ for $\bx \in \II^m$ and 
$ \mu$-almost everywhere for $\by \in \DD^\infty$.
In a typical case, $\sigma(\bx,\by)$ allows for an expansion
$
\sigma(\bx,\by)
  = 
\bar{\sigma}(\bx) + \sum_{j=1}^\infty \psi_j(\bx) y_j, 
$
where $\bar{\sigma} \in L_\infty(\II^m)$ and 
$(\psi_j)_{j=1}^\infty \subset L_\infty(\II^m)$. A choice for $(\psi_j)_{j=1}^\infty$ in sPDEs is the Karh\'unen-Lo\`eve basis where $\bar{\sigma}$ is the average of $\sigma$ and the $y_j$ are pairwise decorrelated random variables.
Another situation is the case, where the logarithm of the diffusion coefficient $\sigma(\bx,\by)$ can be represented by a centered Karh\'unen-Lo\`eve expansion
$\sigma(\bx,\by)
:= \
\exp\left(\sum_{j=1}^\infty y_j \psi_j(\bx)\right). 
$

Note here that $u(\bx,\by)$ can be seen as a map $u(\cdot,\cdot): \DD^\infty \to H^\beta(\II^m)$ of the second variable $\by$. 
Usually, for the elliptic problem (\ref{spde}), we consider the smoothness indices $\beta=0$ or $\beta=1$. 
In general, the solution $u$ lives in the Bochner space
$$
L_p(\DD^\infty,H^\beta(\II^m)) :=\left\{ u: \DD^\infty\to H^\beta(\II^m) : \int_{\DD^\infty} \|u(\bx,\by)\|_{H^\beta(\II^m)}^p \mathrm{d} \mu(\by) < \infty \right\}
$$
(with a natural modification for $p=\infty$), where $H^\beta(\II^m)$ is the Sobolev space of smoothness $\beta$.
For reasons of simplicity, we restrict ourselves to the Hilbert space setting and consider $p=2$. Then, since $H^\beta(\II^m)$ is a Hilbert space as well,
$L_2(\DD^\infty,H^\beta(\II^m))$ is isomorphic to the tensor product space $H^\beta(\II^m) \otimes L_2(\DD^\infty)$
and we can measure $u$ in the associated norm $\|\cdot\|_{H^\beta(\II^m) \otimes L_2(\DD^\infty)}$.

Furthermore, depending on the properties of the diffusion function $\sigma(\bx,\by)$ and the right hand side $f(\bx)$, we have higher regularity of $u$ in both,
$\bx$ and $\by$.
While we directly may assume that $u(\cdot, \by)$ is in $H^{\alpha}(\II^m), \alpha > \beta$, pointwise for each 
$\by$, the regularity of $u$ in $\by$ needs further consideration. 
It is known that, under mild assumptions on $\sigma(\bx,\by)$, the solution of (\ref{spde}) depends analytically on the variables in $\by$, see e.g. \cite{Bab07,BNTT,CDS10b,CDS10a}. 
Moreover, there are estimates that show a mixed-type analytic regularity of $u$ in the $\by$-part, i.e. we indeed have\footnote{Consequently, there is here no issue with our specific choice of norm as in \cite{PW10}.} 
$u(\bx,\cdot) \in A^{\br, p,q}(\DD^\infty)$.
For the simple affine case with a product Legendre expansion, this will be discussed in more detail in the appendix.
  For estimates on the expansion coefficients, see e.g. \cite{CDS10a}, formula (4.9),  \cite{CDS10b}, section 6, or \cite{Chkifathesis}, subsection 1.3.2, for the case of uniformly elliptic diffusion and \cite{HS14} for the case of log-normally distributed diffusion, and see also \cite{BNTT2,TWZ15}.
Analogous estimates and derivations hold (after some tedious calculations) for more complicated non-affine settings and diffusions, see, e.g. \cite{Chkifathesis,TWZ15}, provided that corresponding proper assumptions on $ \sigma(\bx, \by)$ and thus on $u(\bx,\by)$ are valid.  

In addition, and this is less noticed, we also have a mixed-type regularity of $u$ {\em between} the $\bx-$ and the $\by-$part. 
To be precise, a calculation following the lines of \cite{Bab07} which involves successive differentiation of (\ref{spde}) with respect to $\by$ 
reveals that the solution $u$ belongs to the Bochner space $A^{\br,p,q}(\DD^\infty,H^\alpha(\II^m))$ 
provided that $\sigma(\bx,\by)$ and $f(\bx)$ are sufficiently smooth. Since $A^{\br,p,q}(\DD^\infty,H^\alpha(\II^m))$ is isomorphic to $H^\alpha(\II^m) \otimes A^{\br,p,q}(\DD^\infty)$ we indeed have mixed regularity between $\bx$ and $\by$. 
In the end, this is a consequence of the chain rule of differentiation with respect to $\by$ and the structure of the sPDE (\ref{spde}) which involves derivatives with respect to $\bx$ only. If there is not enough smoothness, then $A^{\br,p,q}(\DD^\infty)$ has to be replaced by $K^{\bf r}(\DD^\infty)$ (with some different $\br$), but the mixed regularity structure between $\bx-$ and $\by-$part remains.\footnote{The determination of $\br$  from e.g. the covariance structure of $\sigma(\bx,\by)$ is however not an easy task. In general, a direct functional map from the covariance eigenvalues (if allocatable) to the sequence $\br$ is not available at least to our knowledge, and only estimates can be derived. 
}
Thus,
we first consider the case where $u(\bx,\cdot)$ is in a space of analytic-type smoothness $\br,p,q$ for each $\bx$ with certain smoothness indices $\br,p,q$.
Then, we consider the case where $u(\bx,\cdot)$ is in a space of Korobov-type smoothness $\br$.

In what follows, we keep the notation of Sections \ref{Approximation in Hilbert spaces} and \ref{Function approximation}, and in particular, the notation of Subsection \ref{Nonperiodic approximation}. Let $\alpha > \beta \ge 0$, $a= \alpha - \beta$ and the $5$-tuple $m,a,\br,p,q$ be given as in \eqref{m,a,br,p,q[A]}.
For convenience, we allow, again by a slight abuse of notation, for the identification 
\[
K^{\beta}(\II^m \times \DD^\infty)
\ = \ 
K^{\beta}(\II^m) \otimes L_2(\DD^\infty)
\ = \ 
H^\beta(\II^m) \otimes L_2(\DD^\infty)
\]
and allow for $u$ to belong to the space of analytic-type smoothness 
\[
A^{\alpha,\br,p,q}(\II^m \times \DD^\infty)
\ = \
K^{\alpha}(\II^m) \otimes A^{\br,p,q}(\DD^\infty)
\ = \
H^{\alpha}(\II^m) \otimes A^{\br,p,q}(\DD^\infty)
\]
We then are just in the situations which we analyzed in Subsection \ref{Nonperiodic approximation}.
Let us combine Theorems \ref{theorem[n_e]}, \ref{theorem[E<]} and \ref{theorem[E<](2)} and reformulate them in a more conventional form. 
Taking the hyperbolic cross 
$G^A(T):=G_{\ZZmp \times \ZZips}^A(T)=E(T)$ as in \eqref{E(T)} and \eqref{G_{ZZmp X ZZips}^A}, using the orthogonal projection $\Ss_T^A$ as in \eqref{[S_T]-nonperiodic} and putting 
$n:= |G^A(T)|$, we redefine $\Ss_T^A$ as a linear operator of rank $n$
\[
L_n:  K^{\beta}(\II^m) \otimes L_2(\DD^\infty) \to \Pp^A(T),
\] 
where $\Pp^A(T)$ is defined by \eqref{Pp^A(T)}.
Suppose  that there hold  the assumptions of 
Theorem \ref{theorem[E<]} if $p=0$, and the assumptions of Theorem \ref{theorem[E<](2)} if $p>0$. From Theorems \ref{theorem[n_e]}, \ref{theorem[E<]} and \ref{theorem[E<](2)} we obtain an error bound of the approximation of the solution $u$ by $L_n$ as
\[
\|u - L_n(u)\|_{K^{\beta}(\II^m) \otimes L_2(\DD^\infty)} 
\ \leq \ 
2^{\alpha - \beta} \exp[\,(\alpha - \beta)\,M_{p,q}/m] \,  
n^{- (\alpha - \beta)/m} \, \| u\|_{K^{\alpha}(\II^m) \otimes A^{\br,p,q}(\DD^\infty)}.
\]

Next, we consider the case where $u(\bx,\cdot)$ is only in a space of Korobov-type smoothness $\br$. 
To this end, recall the details of Subsection \ref{Nonperiodic approximation}. 
Again, based on  the orthonormal bases 
$\{\phi_{1,k}\}_{k \in I}$ and $\{\phi_{2,s}\}_{s \in J}$ 
for the two spaces $H_1$ and $H_2$ in (\ref{H1H2}) with 
$I= J=\ZZ_+$ as defined in \eqref{L-polynomials}, respectively,
we construct the Korobov-type spaces 
\[
K^{\beta}(\IIm \times \DDi):= \Gg, \quad K^{\alpha,\br}(\IIm \times \DDi):= \Hh,
\]
where $\Gg$ and $\Hh$ are defined as in \eqref{[GgHh]} with $F= K^{\alpha,\br}$ and the triple 
$m,\alpha,\br$ is from \eqref{m,a,br[K]}.

For $T \ge 0$, we denote by $\Pp^K(T)$ the subspace of polynomials $g$ of the form 
\begin{equation}
g
\ = \ 
\sum_{(\bk,\bs) \in G^K(T)}   g_{\bk,\bs} \phi_{(\bk,\bs)},
\end{equation}
where 
\begin{equation} \label{G_{ZZmp X ZZips}^K}
G^K(T)
\ = \
G_{\ZZmp \times \ZZips}^K(T) 
:= \ 
\big\{(\bk,\bs)  \in \ZZmp \times \ZZips: \lambda_K(\bk,\bs)   \leq T\big\},
\end{equation}
and 
\begin{equation} \nonumber
\lambda_K(\bk,\bs) 
 := \
 \max_{1\le j \le m}(1 + |k_j|)^a \, \prod_{j=1}^\infty (1 +  s_j)^{r_j}.
\end{equation}
For $f \in L_2(\IIm \times \DDi)$ and $T \ge 0$, we define the operator $\Ss_T^K$ as
\begin{equation} \label{[S_T^k]-nonperiodic}
\Ss_T^K(f)
:= \ 
\sum_{(\bk,\bs) \in G^K(T)}  f_{\bk,\bs}\phi_{(\bk,\bs)}.
\end{equation}
Then, we see from Theorem \ref{theorem[n_e]} that for arbitrary $T \geq 1$
\begin{equation} \label{|f - S_T^K(f)|}
\|f - \Ss_T^K(f)\|_{K^{\beta}(\II^m) \otimes L_2(\DD^\infty)}
\ \le  \
T^{-1}\|f\|_{K^{\alpha}(\II^m) \otimes K^{\br}(\DD^\infty)} \, , \qquad  \forall f \in K^{\alpha,\br}(\IIm \times \DDi).
\end{equation}
On the other hand, let $\alpha > \beta \ge 0$, let $a= \alpha - \beta$ and let the triple 
$m,a,\br$ be given as in \eqref{m,a,br[K]}.
Suppose  that there hold the assumptions of Theorem \ref{theorem[G<]} and moreover, $r > (\alpha - \beta)/m$.
Then, we have by Theorem \ref{theorem[G<]} for every $T \geq 1$,
\begin{equation} 
|{G^K}(T)|
\ \le \ 
C\, T^{m/(\alpha-\beta)}, 
\end{equation}
where $C:= C(a,r,m,t)$ is as in \eqref{C(a,r,m,t)} 
(and with $G^K(T)= G(T)$ in \eqref{G(T)}). 
Setting $n:= |G^K(T)|$, we redefine the orthogonal projection $\Ss_T^K$ as a linear operator of rank $n$
\[
L_n: K^{\beta}(\II^m) \otimes L_2(\DD^\infty) \to  \Pp^K(T).
\] 
From \eqref{|f - S_T^K(f)|} and 
Theorem \ref{theorem[G<]}, we obtain an error bound of the approximation of $u$ by $L_n$ as
\[
\|u - L_n(u)\|_{K^{\beta}(\II^m) \otimes L_2(\DD^\infty)} 
\ \leq \ 
C^{(\alpha - \beta)/m} \,  
n^{- (\alpha - \beta)/m} \, \| u\|_{K^{\alpha}(\II^m) \otimes K^{\br}(\DD^\infty)}.
\]
Note finally that the results in this section can be extended without difficulty to the periodic setting or to mixed periodic and nonperiodic settings of \eqref{spde}.

\section{Concluding remarks}\label{conclusion}

In this article we have shown how the determination of the $\varepsilon$-dimension for the approximation of infinite-variate function classes with anisotropic mixed smoothness can be reduced to the problem of tight bounds of the cardinality of associated hyperbolic crosses in the infinite-dimensional case.
Moreover, we explicitly computed such bounds for a range of function classes and spaces.
 Here, the approximation was based on linear information. 
The obtained upper and lower bounds of the $\varepsilon$-complexities as well as the  
convergence rates of the associated approximation error
 are completely independent of any parametric or stochastic dimension provided that moderate and quite natural summability conditions on the smoothness indices of the underlying infinite-variate spaces are valid.  These parameters are only contained in the order constants. 
This way, linear approximation theory becomes possible in the infinite-dimensional case and corresponding infinite-dimensional problems get manageable.  

For the example of the approximation of the solution of an elliptic stochastic PDE it indeed turned out that the infinite-variate stochastic part of the problem has completely disappeared from the cost complexities and the convergence rates and influences only the constants.
Hence, these problems are strongly polynomially tractable (see \cite{NW08} for a definition).
Note at this point that the
$m$-variate physical part of the problem and the infinite-variate stochastic part are not separately treated in our analysis but are {\em collectively} approximated where the hyperbolic cross approximation involves a simultaneous projection onto both parts which profits from the mixed regularity situation and the corresponding product construction.
Here, we restricted ourselves to a Hilbert space setting and to linear information. Furthermore, we considered an a priori, linear approximation approach.

We believe that our analysis can be generalized to the Banach space situation, in particular, to the $L_1$- and $L_\infty$-setting 
which is related to problems of interpolation, integration and collocation. 
Then, instead of linear information, standard information via point values is employed and non-intrusive techniques can be studied, 
which are widely used in practice. 
To this end, the efficient approximative computation of the coefficients $f_{\bk,\bs}$ still needs to be investigated and analyzed in detail. 
We hope that some of the ideas and techniques presented in the this paper will be useful there.

For our analysis, we  assumed the a priori knowledge of the smoothness indices and their monotone ordering. This is sound if these smoothness indices stem from an 
eigenvalue analysis of the covariance structure of the underlying problem and are explicitly known or at least computable.
If this is however not the case, then, instead of our a priori definition of the hyperbolic crosses from the smoothness indices, we may generate suitable sets of active indices in an a posteriori fashion by means of dimension-adaptivity in a way which is similar to dimension-adaptive sparse grid methods \cite{GGr,GOe}.    
  
Finally, recall that we assume linear information and an associated cost model which assigns a cost of $O(1)$ to each evaluation of a coefficient $f_{\bk,\bs}$.
Or, the other way around, we just count each index $(\bk,\bs)$ in a hyperbolic cross as one.
We may also consider more refined cost models which take into account that the number $\#(\bk,\bs)$ of non-zero entries of an index $(\bk,\bs)$ is always finite.
This would allow to relate the cost of an approximation of the coefficient $f_{\bk,\bs}$ to the number of non-zero entries
of each $(\bk,\bs)$ in a hyperbolic cross induced by $T$. For examples of such refined cost models, see e.g. the discussion in \cite{DKS} and the references cited therein.
In our case, this would lead to a cost analysis where the cardinality of the respective hyperbolic cross is not just counted by adding up ones
in the summation, but by instead adding up values which depend on $\#(\bk,\bs)$ via e.g. a function thereof, which reflects the respective cost model.

\section*{Acknowledgments}
Dinh Dung's research work was funded by Vietnam National Foundation for Science and Technology Development (NAFOSTED) under Grant No. 102.01-2014.02.
Michael Griebel was partially supported by the Sonderforschungsbereich 1060 {\em The Mathematics of
Emergent Effects} funded by the Deutsche Forschungsgemeinschaft.
The authors thank the Institute for Computational and Experimental Research in Mathematics (ICERM) at Brown University for its hospitality and additional support during the preparation of this manuscript. They thank Alexander Hullmann, Christian Rieger and Jens Oettershagen for valuable discussions. 

\section*{Appendix}
We now show for a simple model problem of the type (\ref{spde}) that its solution $u$ belongs to the space $ K^\alpha(\II^m) \otimes A^{\br,p,q}(\DD^\infty)$ for any $p,q \geq 0$
provided that a certain assumption on $\sigma$ and thus a certain condition on $\br$ is satisfied. 
Consequently, our theory is indeed applicable here.

We consider the problem 
 \begin{align}\label{spde2}
-\mathrm{div}_\bx \left(\sigma(\bx,\by) \nabla_\bx u(\bx,\by) \right) = f(\bx), \quad u|_{\partial \mathbb{I}^{m}}=0.
\end{align}
Here, we assume that 
$0<\sigma_{\min} \le \sigma(\bx,\by) \le \sigma_{\max}<\infty$ for all $\bx \in \mathbb{I}^{m}$ and all $\by \in \mathbb{D}^{\infty}$.
Moreover, we follow closely the seminal article \cite{CDS10b} and consider the linear affine setting
\begin{equation}\label{afflinear}
\sigma(\bx,\by)= \bar {\sigma}(\bx)+\sum_{j=1}^\infty \psi_{j}(\bx) y_{j}.
\end{equation}
We furthermore assume that there is $0<\beta<\alpha<\infty$ and $a:=\alpha-\beta$ such that\footnote{Especially, we even may assume here $K^\beta \simeq H_{0}^{1}(\mathbb{I}^{m})$, i.e. $\beta =1$, and $K^\alpha \simeq H^2(\mathbb{I}^{m}) \cap H_{0}^{1}(\mathbb{I}^{m})$, i.e. $\alpha =2$, and thus $a=1$.}
 \begin{align*}
 H_{0}^{1}(\mathbb{I}^{m}) \hookrightarrow K^{\beta} \quad \text{and} \quad \left\{ w \in H_{0}^{1}(\mathbb{I}^{m})\ : \ \Delta w \in L^{2}(\mathbb{I}^{m}) \right\} \hookrightarrow K^{\alpha} .
 \end{align*}
Recall our definitions (\ref{H1H2}) and (\ref{def[g-presentation]}) of the spaces $\Ll$ and $\Ll^\lambda$.
Also recall our definition (\ref{[lambda{a,mu,br}]}) of the scalar
 \begin{align}\label{weight}
 \rho(\bk,\bs):=\max_{1\le j \le m}\left(1+|k_{j}| \right)^{a} \prod_{j=1}^{\infty}\left( 1+p|s_{j}|\right)^{-q} \exp\left(\sum_{j=1}^{\infty} r_{j} |s_{j}| \right)
 =: \lambda_{m,a}(\bk)\rho_{\br,p,q}(\bs)=:\lambda_{A}(\bk,\bs)
 \end{align}
together with the notations of $\lambda_{m,a}(\bk)$ from (\ref{defa_index}), $\rho(\br,p,q)$ from (\ref{[rho^br,pq]}) and $\lambda_{A}(\bk,\bs)$ 
from (\ref{def_lam_A}), respectively. We encounter the non-periodic setting here and thus have only non-negative values for the indices $\bk$ and $\bs$ in the sets 
$\Ii,\Jj$, i.e. $(\bk,\bs) \in \ZZ^m_+ \times \ZZ^\infty_+$.
Now, for the orthogonal basis $ \phi_{\bk,\bs}:=\phi_{1,\bk} \otimes \phi_{2,\bs}$ of (\ref{oxxxxx}) for $\Ll $ from (\ref{H1H2}), we assume for reasons of simplicity some suitable orthogonal basis for $\phi_{1,\bk}$ and we specifically use products of Legendre polynomials for $\phi_{2,\bs}$. 
We set 
\begin{align*}
\phi_{2,n}(y)=\frac{(-1)^{n}\sqrt{2n+1} }{2^{n}n!} \frac{d^{n}}{dy^{n}} \left(1-y^{2}\right)^{n}
\end{align*}
and $
\phi_{2,\bs}=\bigotimes_{\genfrac{}{}{0pt}{}{j\in \mathcal{J}}{s_{j}\neq 0}} \phi_{2,s_{j}},
$ c.f. (\ref{ozzzzz}).
Furthermore, we have that 
\begin{align*}
\int_{-1}^{1} \phi_{2,s_{j}}(y) \phi_{2,s_{k}}(y) \frac{dt}{2} = \delta_{k,s}, \quad \text{and} \quad \max_{-1 \le y \le 1} \left| \phi_{2,s_{j}}(y) \right| = \sqrt{2s_{j}+1} .
\end{align*}
We consider a series representation of the solution of (\ref{spde2}) as
\begin{align*}
u(\bx,\by)=\sum_{\bs \in \mathcal{J}} \hat{u}_{2,\bs}(\bx) \phi_{2,\bs}(\by).
\end{align*}
For our simple affine case with product Legendre expansion, explicit estimates for the expansion coefficients $ \hat{u}_{2,\bs}(\bx) $ can be found in \cite{CDS10b}, section 6, or \cite{Chkifathesis}, subsection 1.3.2. 
To be precise, we can use Corollary 6.1. of \cite{CDS10b} to obtain
\begin{align} \label{nonprodest}
\left\| \hat{u}_{2,\bs}\right\|_{H_{0}^{1}(\mathbb{I}^{m})} \le  \left(\frac{\left\|f \right\|_{H^{-1}(\mathbb{I}^{m})} }{\sigma_{\min}} \right) \frac{\left| \bs \right|!}{\bs!} \prod_{j=1, s_{j}\neq 0}^\infty \left(\frac{\left\|\psi_{j} \right\|_{L^{\infty}(\mathbb{I}^{m})} }{\sqrt{3}\sigma_{\min}} \right)^{s_{j}}=:B  \frac{\left| \bs \right|!}{\bs !} \bb^{\bs}.
\end{align}
Next, we have
\begin{align*}
\rho_{\br,p,q}(\bs)\left\| \hat{u}_{2,\bs}\right\|_{H_{0}^{1}(\mathbb{I}^{m})}&= \prod_{j=1, s_j \neq 0}^{\infty}\left( 1+ps_{j}\right)^{-q} \exp\left(\sum_{j=1}^{\infty} r_{j} s_{j} \right)\left\| \hat{u}_{2,\bs}\right\|_{H_{0}^{1}(\mathbb{I}^{m})}\\
&\le B \frac{ \left| \bs \right|!}{\bs !}  \prod_{j=1, s_j \neq 0}^{\infty}\left(  \left( 1+ps_{j}\right)^{-q} \exp\left(r_{j} s_{j} \right) b_{j}^{s_{j}} \right).
\end{align*}
Moreover, since $\left( 1+ps_{j}\right)^{-q} \leq 1$, this yields
\begin{align*}
\rho_{\br,p,q}(\bs)\left\| \hat{u}_{2,\bs}\right\|_{H_{0}^{1}(\mathbb{I}^{m})} \le B \frac{\left| \bs \right|!}{\bs !}  \prod_{j=1}^{\infty} \exp\left(r_{j} s_{j} \right) b_{j}^{s_{j}} =  B \frac{\left| \bs \right|!}{\bs !}    \prod_{j=1}^{\infty} \left(\exp\left( r_{j}\right)b_j \right)^{s_{j}}=: B \frac{\left| \bs \right|!}{\bs !} \tilde{\bb}(\br)^{\bs}
\end{align*}
with 
\begin{align*}
\tilde{b}_{j}(\br):= \frac{\left\|\psi_{j} \right\|_{L^{\infty}(\mathbb{I}^{m})} }{\sqrt{3}a_{\min}}\exp\left( r_{j}\right).
\end{align*}
Now, we have
\begin{align*}
&\left\|u \right\|_{K^\alpha(\mathbb{I}^{m}) \otimes A^{\br,p,q}(\mathbb{D}^{\infty}) } \simeq \left\|u \right\|_{ A^{\br,p,q}(\mathbb{D}^{\infty},K^{\alpha}(\mathbb{I}^{m})) }
= \left\|\left(\left\| \hat{u}_{2,\bs}\right\|_{K^{\alpha}(\mathbb{I}^{m})} \right)_{\bs \in \mathcal{J}} \right\|_{\ell^{2}(\mathcal{J})} \le \left\|\left( B \frac{\left| \bs \right|!}{\bs !} \tilde{\bb}(\br)^{\bs} \right)_{\bs \in \mathcal{J}} \right\|_{\ell^{2}(\mathcal{J})} 
\end{align*}
and we want to derive a condition for it to be finite.
To this end, we have
\begin{align*}
\left\|\left( B \frac{\left| \bs \right|!}{\bs !} \tilde{\bb}(\br)^{\bs} \right)_{\bs \in \mathcal{J}} \right\|_{\ell^{2}(\mathcal{J})}&\le 
\left\|\left( B \frac{\left| \bs \right|!}{\bs !} \tilde{\bb}(\br)^{\bs} \right)_{\bs \in \mathcal{J}} \right\|_{\ell^{\infty}(\mathcal{J})}\left\|\left( B \frac{\left| \bs \right|!}{\bs !} \tilde{\bb}(\br)^{\bs} \right)_{\bs \in \mathcal{J}} \right\|_{\ell^{1}(\mathcal{J})}
\le \left\|\left( B \frac{\left| \bs\right|!}{\bs !} \tilde{\bb}(\br)^{\bs} \right)_{\bs \in \mathcal{J}} \right\|^{2}_{\ell^{1}(\mathcal{J})}.
\end{align*}
Now, we can apply Theorem 7.2. of \cite{CDS10b} and get
\begin{align*}
&\left\|\left( B \frac{\left| \bs \right|!}{\bs !} \tilde{\bb}(\br)^{\bs} \right)_{\bs \in \mathcal{J}} \right\|_{\ell^{1}(\mathcal{J})}=\sum_{\bs \in \mathcal{J}}\left|B \frac{\left| \bs\right|!}{\bs !} \tilde{\bb}(\br)^{\bs}\right|< \infty \\
&\Longleftrightarrow
\left\|\left(\tilde{b}_{j}(\br) \right)_{j \in \NN} \right\|_{\ell^{1}(\mathbb{N})}=\sum_{j=1}^{\infty}\frac{\left\|\psi_{j} \right\|_{L^{\infty}(\mathbb{I}^{m})} }{\sqrt{3}\sigma_{\min}}\exp\left( r_{j}\right) <1.
\end{align*}
Furthermore, we obtain in this case
\begin{align*}
\left\|\left( B \frac{\left| \bs\right|!}{\bs !} \tilde{\bb}(\br)^{\bs} \right)_{\bs \in \mathcal{J}} \right\|_{\ell^{1}(\mathcal{J})}=\sum_{\bs \in \mathcal{J}}\left|B \frac{\left| \bs \right|!}{\bs !} \tilde{\bb}(\br)^{\bs}\right| = \frac{B}{1-\left\|\left(\tilde{b}_{j}(\br) \right)_{j \in \NN} \right\|_{\ell^{1}(\mathbb{N})}}.
\end{align*}
Thus, it finally holds that
\begin{align*}
&\left\|u \right\|^2_{K^{\alpha}(\mathbb{I}^{m}) \otimes A^{\br,p,q}(\mathbb{D}^{\infty}) } \simeq \left\|\left(\left\| \hat{u}_{2,\bs}\right\|_{K^{\alpha}(\mathbb{I}^{m})} \right)_{\bs \in \mathcal{J}} \right\|^2_{\ell^{2}(\mathcal{J})} \le \left\|\left( B \frac{\left| \bs \right|!}{\bs !} \tilde{\bb}(\br)^{\bs} \right)_{\bs \in \mathcal{J}} \right\|^2_{\ell^{2}(\mathcal{J})} < \infty 
\end{align*}
is equivalent to
\begin{align}
\frac{1}{\sqrt{3}\sigma_{\min}}\sum_{j=1}^{\infty}\left\|\psi_{j} \right\|_{L^{\infty}(\mathbb{I}^{m})} \exp\left( r_{j}\right) <1 . \label{conditionrrr}
\end{align}
In other words, we have derived that $u\in K^\alpha(\II^m) \otimes A^{\br,p,q}(\DD^\infty)$ for any $p,q \geq 0$
provided that the condition (\ref{conditionrrr}) is satisfied.

For the case $p,q >0$, it remains to check the condition (\ref{Mpq}) from Theorem \ref{theorem[E<](2)}.
If we assume $r_{j}> pq$ and $r_j \geq \frac {q+\sqrt{qa/m}}{p}$, then we have with $a=\alpha-\beta$ that
\begin{align*}
M_{p,q}(m):=(1+p/2)^{qm/a}\sum_{j=1}^{\infty} \frac{\exp\left(-\frac{m}{2a} r_{j} \right)}{r_{j} \frac{m}{a}-pq \frac{m}{a}}
\end{align*}
is finite if
\begin{align}\label{Mpq2}
\sum_{j=1}^{\infty} \frac{\exp\left(-\frac{m}{2a} r_{j} \right)}{r_{j} \frac{m}{a}-pq \frac{m}{a}} \le \frac 1 {r_1 -pq}  \frac a m \sum_{j=1}^{\infty} \exp\left(-\frac{m}{2a} r_{j} \right)
< \infty .
\end{align}
Hence, with suitable constant $ c$, we have for any bounded $\sigma(\bx,\by)$ of the form (\ref{afflinear}) with
\begin{align*}
\left\|\psi_{j} \right\|_{L^{\infty}(\mathbb{I}^{m})} \le  c \cdot \exp \left(-\left( 1+\frac{m}{2a}\right) r_{j}\right) 
\end{align*}
that both conditions, i.e. (\ref{conditionrrr}) and (\ref{Mpq2}), are fulfilled.

Analogously, for the case $p=0,q \geq 0$,  it remains to check the condition (\ref{M00}) from Theorem \ref{theorem[E<]}.
Then we have that
\begin{align*}
M_{0,q}(m):=\sum_{j=1}^{\infty} \frac{1}{e^{mr_j/a}-1} 
\end{align*}
is finite if 
\begin{align}\label{M002}
\sum_{j=1}^{\infty} \frac{1}{e^{mr_j/a}-1} \le \frac 1{1-e^{-{mr_1/a}}} \sum_{j=1}^{\infty} \exp\left(-\frac{m}{a} r_{j} \right) < \infty .
\end{align}
Hence, with suitable constant $ c$, we have for any bounded $\sigma(\bx,\by)$ of the form (\ref{afflinear}) with
\begin{align*}
\left\|\psi_{j} \right\|_{L^{\infty}(\mathbb{I}^{m})} \le c \cdot \exp \left(-\left( 1+\frac{m}{a}\right) r_{j}\right) 
\end{align*}
that both conditions, i.e. (\ref{conditionrrr}) and (\ref{M002}), are fulfilled.
Moreover, in this case we do not encounter the factor $1+m/(2a)$ in the exponent but merely the improved factor $1+m/a$. 

Let us finally mention that, for our affine case with a product Legendre expansion and additionally using a $\delta$-admissibility condition (c.f. \cite{CDS10a}, formula (2.8)), 
there are explicit estimates for the corresponding expansion coefficients in \cite{CDS10a}, subsection 4.2, or \cite{Chkifathesis}, subsection 1.3.2, see also \cite{BNTT2}, proposition 7. 
In contrast to (\ref{nonprodest}), these bounds now have product structure. 
They match (\ref{[rho^br,pq]}) with associated values $\br$ and $p=2,q=1/2$
up to an $\br$-dependent product-type prefactor. It has the form $\prod_{j=1,s_j \neq 0}^\infty \phi(e^{r_j})$ with $\phi(t)=  \frac {\pi t} {2(t-1)}$ and looks independent of $s_j$ at first sight. But, due to the condition $s_j \neq 0$, it is indeed dependent on $\bs$.
After some calculation, it can be shown that there exist a modified sequence $\tilde \br$ and modified $\tilde p,\tilde q$ such that (\ref{[rho^br,pq]}) with, for example, the values $\tilde \br = \br-\frac{1+\varepsilon} 2 \log(j), \varepsilon > 0$ and $\tilde p=2,\tilde q=3/2$ is exactly matched. 
It is now a straightforward calculation to derive 
$u(\bx,\cdot) \in A^{\tilde \br,2,3/2}(\DD^\infty)$ 
due to its definition via 
(\ref{def[Kbr]}) 
and 
(\ref{norm[Kbr]}). 
 Of course, it remains to show that $M_{\tilde p,\tilde q}(m) < \infty$ is valid for these $\tilde \br$, compare (\ref{Mpq}).
To this end, we obtain for $\sum_{j=1}^\infty \exp(-\frac m {2a} \tilde r_j)/(\tilde r_j-\tilde p \tilde q)$  the upper bound $\sum_{j=1}^\infty \exp(-\frac m {2a} r_j) j^{\frac{m(1+\varepsilon)}{4a}}$ (up to a constant).
Thus, for example in the case $m=3,a=1$, a growth of $\br$  like $r_j\geq ({7/6+\varepsilon}) \log(j), \varepsilon > 0$ is sufficient. 
Note that this is indeed a quite mild condition.

\end{document}